\documentclass{amsart}
\usepackage{amssymb}
\newtheorem{theorem}{Theorem}[section]
\newtheorem{lemma}[theorem]{Lemma}

\newtheorem{remark}[theorem]{Remark}
\newtheorem{definition}[theorem]{Definition}

\newtheorem{corollary}[theorem]{Corollary}
\newtheorem{proposition}[theorem]{Proposition}
\newcommand{\R}{\mathbb R}

\newcommand{\Z}{\mathbb Z}
\newcommand{\Q}{\mathbb Q}

\def\op{\operatorname}
\def\a{\mathfrak{a}}

\def\as#1{\renewcommand\arraystretch{#1}}

\def\b{\mathfrak{b}}
\def\bb{{\mathcal B}}
\def\be{\bigskip}

\def\dg{\op{deg}}

\def\dsc{\op{disc}}
\def\diso{\lower.4ex\hbox{$\downarrow$}\raise.4ex\hbox{\mbox{\scriptsize
$\wr$}}}

\def\exp{\op{exp}}
\def\ff#1{\mathbb{F}_{#1}}

\def\ga{\gamma}
\def\gb#1{\overline{\gamma_{#1}(\t_\p)}}
\def\gen#1{\big\langle\, {#1} \,\big\rangle}

\def\iso{\ \lower.3ex\hbox{\as{.08}$\begin{array}{c}\lra\\\mbox{\tiny $\sim\,$}\end{array}$}\ }

\def\l{\mathfrak{l}}
\def\la{\lambda}

\def\lg{l\raise.6ex\hbox to.2em{\hss.\hss}l}
\def\lra{\longrightarrow}
\def\m{{\mathfrak m}}
\def\md#1{\ \mbox{\rm(mod }{#1})}

\def\nn{\noindent}

\def\orb{\hbox to  .3em{$\backslash$}\backslash}
\def\ord{\op{ord}}
\def\p{\mathfrak{p}}
\def\P{\mathfrak{P}}

\def\pp{{\mathcal P}}

\def\q{\mathfrak{q}}
\def\qb{\overline{\mathbb{Q}}}
\def\qp{\mathbb{Q}_p}
\def\qpb{\overline{\mathbb{Q}}_p}

\def\sii{\,\Longleftrightarrow\,}

\def\t{\theta}
\def\tb{\overline{\theta_\p}}
\def\ty{\mathbf{t}}

\def\tq{\,\,|\,\,}

\def\zpx{\Z_p[x]}

\newcounter{cs}
\stepcounter{cs}
\newcommand{\casos}{\begin{itemize}}
\newcommand{\fcasos}{\end{itemize}\setcounter{cs}{1}}

\newfont{\tit}{cmr12 scaled \magstep3}

\begin{document}
\title{A new computational approach to ideal theory in number fields}
\author[Gu\`ardia]{Jordi Gu\`ardia}
\address{Departament de Matem\`atica Aplicada IV, Escola Polit\`ecnica Superior
d'Enginyera de Vilanova i la
Geltr\'u, Av. V\'\i ctor Balaguer s/n. E-08800 Vilanova i la Geltr\'u,
Catalonia}
\email{guardia@ma4.upc.edu}

\author[Montes]{\hbox{Jes\'us Montes}}
\address{Departament de Ci\`encies Econ\`omiques i Socials,
Facultat de Ci\`encies Socials,
Universitat Abat Oliba CEU,
Bellesguard 30, E-08022 Barcelona, Catalonia, Spain}
\email{montes3@uao.es}

\author[Nart]{\hbox{Enric Nart}}
\address{Departament de Matem\`{a}tiques,
         Universitat Aut\`{o}noma de Barcelona,
         Edifici C, E-08193 Bellaterra, Barcelona, Catalonia, Spain}
\email{nart@mat.uab.cat}
\thanks{Partially supported by MTM2009-13060-C02-02 and MTM2009-10359 from the
Spanish MEC}
\date{}

\makeatletter
\@namedef{subjclassname@2010}{%
  \textup{2010} Mathematics Subject Classification}


\begin{abstract}
Let $K$ be the number field determined by a monic irreducible polynomial $f(x)$ with integer coefficients. In previous papers we parameterized the prime ideals of $K$ in terms of certain invariants attached to Newton polygons of higher order of the defining equation $f(x)$. In this paper we show how to carry out the basic operations on fractional ideals of $K$ in terms of these constructive representations of the prime ideals. From a computational perspective, these
results facilitate the manipulation of fractional ideals of $K$ avoiding two heavy tasks: the construction of the maximal order of $K$ and the factorization of the discriminant of $f(x)$. The main computational ingredient is Montes algorithm, which is an extremely fast procedure to construct the prime ideals.
\end{abstract}

\maketitle
%\tableofcontents
\section*{Introduction}
Let $K$ be a number field of degree $n$ and $\Z_K$ its ring of integers. An essential task in Algorithmic Number Theory is to construct the prime ideals of $K$ in terms of a defining equation of $K$, usually given by  a monic and irreducible polynomial $f(x)\in\Z[x]$. The standard approach to do this, followed by most of the algebraic manipulators like {\tt Kant}, {\tt Pari}, {\tt Magma} or {\tt Sage}, is based on the previous computation of
an integral basis of $\Z_K$. This approach has a drawback: one needs to factorize the discriminant, $\dsc(f)$, of $f(x)$, which can be a heavy task, even in number fields of low degree, if $f(x)$ has large coefficients.

In this paper we present a direct construction of the prime ideals, that avoids the computation of the maximal order of $K$ and the factorization of $\dsc(f)$. The following tasks concerning fractional ideals can be carried out using this construction:

\begin{enumerate}
\item Compute the $\p$-adic valuation, \ $v_\p\colon K^*\to \Z$, for any prime ideal $\p$ of $K$.
\item Obtain the prime ideal decomposition of a fractional ideal.
\item Compute a two-element representation of a fractional ideal.
\item Add, multiply and intersect fractional ideals.
\item Compute the reduction maps, $\Z_K\to \Z_K/\p$.
\item Solve Chinese remainders problems.
\end{enumerate}

Moreover, along the construction of a prime ideal $\p$, lying over a prime number $p$, a $\Z_p$-basis of the ring of integers of the local field $K_\p$ is obtained as a by-product. Hence, from the prime ideal decomposition of the ideal $p\Z_K$ we are also able to derive the resolution of another task:\medskip
\begin{enumerate}
 \item[(7)] Compute a $p$-integral basis of $K$.
\end{enumerate}
For a given prime number $p$, the prime ideals of $K$ lying above $p$ are in one-to-one correspondence with the irreducible factors of $f(x)$ in $\Z_p[x]$ \cite{hensel}. In \cite{HN} we proved a series of recurrent generalizations of Hensel lemma, leading to a constructive procedure to obtain a family of \emph{$f$-complete types}, that parameterize the irreducible factors of $f(x)$ in $\Z_p[x]$. A type is an object that gathers combinatorial and arithmetic data attached to Newton polygons of $f(x)$ of higher order, and an $f$-complete type contains enough information to single out a $p$-adic irreducible factor of $f(x)$.
In \cite{GMNalgorithm} we described Montes algorithm, which optimizes the construction of the $f$-complete types; it outputs a list of $f$-complete and optimal types that parameterize the prime ideals of $K$ lying above $p$, and contain valuable arithmetic information on each prime ideal. All these results were based on the PhD thesis of the second author \cite{montes}. The algorithm is extremely fast in practice; its complexity has been recently estimated to be $O(n^{3+\epsilon}\delta+n^{2+\epsilon}\delta^{2+\epsilon})$, where $\delta=\log(\dsc(f))$ \cite{FV}.

In \cite{GMNokutsu} we reinterpreted the invariants stored by the types in terms of the Okutsu polynomials attached to the $p$-adic irreducible factors of $f(x)$ \cite{Ok}. Suppose $\ty$ is the $f$-complete and optimal type attached to a prime ideal $\p$, corresponding to a monic irreducible factor $f_\p(x)\in\Z_p[x]$; then, the arithmetic information stored in $\ty$ is synthesized by two invariants of $f_\p(x)$: an \emph{Okutsu frame} $[\phi_1(x),\dots,\phi_r(x)]$ and a \emph{Montes approximation} $\phi_\p(x)$
(cf. loc.cit.). The monic polynomials $\phi_1,\dots,\phi_r,\phi_\p$ have integer coefficients and they are all irreducible over $\Z_p[x]$; the polynomial $\phi_\p(x)$ is ``sufficiently close" to $f_\p(x)$. We say that
$$
\p=[p;\phi_1,\dots,\phi_r,\phi_\p],
$$
is the \emph{Okutsu-Montes representation} of the prime ideal $\p$. Thus, from the computational point of view, $\p$ is structured in $r+1$ levels and at each level one needs to compute (and store) several Okutsu invariants that are omitted in this notation. This computational representation of $\p$ is essentially canonical: the Okutsu inva\-riants of $\p$, distributed along the different le\-vels, depend only on the defining equation $f(x)$. These invariants provide a rich and exhaustive source of information about the arithmetic properties of $\p$, which is crucial in the computational treatment of $\p$.

From a historical perspective, the sake for a constructive representation of ideals goes back to the very foundation of algebraic number theory. Kummer had the
insight that the prime numbers factorize in number fields
into the product of prime ``ideal numbers", and
he tried to construct them as symbols $[p\,;\phi]$, where
$\phi(x)$ is a monic lift to $\Z[x]$ of an irreducible factor of $f(x)$ modulo $p$. Dedekind
showed that these ideas led to a coherent theory only in the case
that $p$ does not divide the index $i(f):=(\Z_K\colon
\Z[x]/(f(x)))$.
This constructive approach could not be universally
used because there are number fields in which $p$ divides the
index of all defining equations \cite{D}. Fortunately, this
obstacle led Dedekind to invent ideal theory as the only way to
perform a decent arithmetic in number fields. Ore, in his Phd
thesis \cite{ore}, tried to regain the constructive approach to
ideal theory. He generalized and improved the classical tool of Newton polygons
and showed that under the assumption that the defining
equation is \emph{$p$-regular} (a much weaker condition than
Dedekind's condition $p\nmid i(f)$), the prime ideals dividing $p$ can
be parameterized as $\p=[p\,;\phi,\phi_\p]$ (in our notation), where
$\phi_\p(x)\in\Z[x]$ is certain polynomial whose $\phi$-Newton polygon
is one-sided and the residual polynomial attached to this side is irreducible (cf. section \ref{secMontes}). The contribution of
\cite{montes} was to extend Ore's ideas in order to obtain a similar construction of the prime ideals in the general case.

The aim of this paper is to show how to use this constructive representation of the prime ideals to carry out the above mentioned tasks (1)-(6) on fractional ideals and to compute $p$-integral bases.
The outline of the paper is as follows. In section \ref{secMontes} we recall the structure of types, we describe their invariants, and we review the process of construction of the Okutsu-Montes representations of the prime ideals. In section \ref{secPadic} we
show how to compute the $\p$-adic valuation of $K$ with respect to a prime ideal $\p$; this is the key ingredient to obtain the factorization of a fractional ideal as a product of prime ideals (with integer exponents). The operations of sum, multiplication and intersection of fractional ideals are trivially based on these tasks. In section \ref{secGenerators} we show how to find integral elements $\alpha_\p\in\Z_K$ such that $\p$ is the ideal of $\Z_K$ generated by $p$ and $\alpha_\p$; this leads to the computation of a two-element representation of any fractional ideal. In section \ref{secCRT}, we show how to compute residue classes modulo prime ideals and we design a chinese remainder theorem routine. Section \ref{secBasis} is devoted to the  construction of a $p$-integral basis.

We have implemented a package in {\tt Magma} that performs all the above mentioned tasks; in section \ref{secKOM} we present several examples showing the excellent performance of the package in cases that the standard packages cannot deal with. Our routines work extremely fast as long as we deal only with fractional ideals whose norm may be factorized. Even in cases where $\dsc(f)$ may be factorized and an integral basis of $\Z_K$ is available, our methods work faster than the standard ones if the degree of $K$ is not too small. Mainly, this is due to the fact that we avoid the use of linear algebra routines (computation of $\Z$-bases of ideals, Hermite and Smith normal forms of $n\times n$ matrices, etc.), that dominate the complexity when the degree $n$ grows. Finally, in section \ref{secConclusion} we make some comments on the apparent limits of these Montes' techniques: they are not yet able to test if a fractional ideal is principal.
We also briefly mention how to extend the results of this paper to the function field case and the similar challenges that arise in this geometric context.\bigskip

%The prime divisors of a curve over a finite field are identified to prime ideals of the integral closures of certain subrings of the function field, and Montes algorithm computes similar constructive representations of these prime ideals. A similar challenge arises: we speculate with the possibility to use these Montes' techniques to find better routines to compute bases of the Riemann-Roch spaces and to deal wih reduced divisors.

\noindent{\bf Notations. }
Throughout the paper we fix a monic irreducible polynomial $f(x)\in\Z[x]$ of degree $n$, and a root $\t\in\qb$ of $f(x)$. We let $K=\Q(\t)$ be the number field generated by $\t$, and $\Z_K$ its ring of integers.

\section{Okutsu-Montes representations of prime ideals}\label{secMontes}
Let $p$ be a prime number. In this section we recall Montes algorithm and we describe the structure of the $f$-complete and optimal types that parameterize the prime ideals of $K$ lying over $p$. The results are mainly extracted from \cite{HN} (HN standing for ``Higher Newton") and \cite{GMNalgorithm}.

Given a field $F$ and  two polynomials $\varphi(y),\,\psi(y)\in F[y]$, we write $\varphi(y)\sim \psi(y)$ to indicate that  there exists a constant $c\in F^*$ such that $\varphi(y)=c\psi(y)$.

\subsection{Types and their invariants}
Let $v\colon \qpb^{\,*}\to \Q$ be the canonical extension of the $p$-adic valuation of $\qp$ to a fixed algebraic closure. We extend $v$ to the discrete valuation $v_1$ on the field $\qp(x)$, determined by:
$$
v_1\colon \Q_p[x]\lra \Z\cup\{\infty\},\quad v_1(b_0+\cdots+b_rx^r):=\min\{v(b_j),\,0\le j\le r\}.
$$

Denote by $\ff0:=\op{GF}(p)$ the prime field of characteristic $p$, and consider the $0$-th \emph{residual polynomial} operator
$$
R_0\colon \zpx\lra \ff0[y],\quad g(x)\mapsto \overline{g(y)/p^{v_1(g)}},
$$
where, $^{\raise.8ex\hbox to 8pt{\hrulefill }}\colon \Z_p[y]\to \ff0[y]$, is the natural reduction map.
A \emph{type of order zero}, $\ty=\psi_0(y)$, is just a monic irreducible polynomial $\psi_0(y)\in\ff0[y]$. A \emph{representative} of $\ty$ is any monic polynomial $\phi_1(x)\in\Z[x]$ such that $R_0(\phi_1)=\psi_0$. The pair $(\phi_1,v_1)$ can be used to attach a Newton polygon to any nonzero polynomial $g(x)\in \qp[x]$. If $g(x)=\sum_{s\ge0}a_s(x)\phi_1(x)^s$ is the $\phi_1$-adic development of $g(x)$, then
$N_1(g):=N_{\phi_1,v_1}(g)$ is the lower convex envelope of the set of points of the plane with coordinates $(s,v_1(a_s(x)\phi_1(x)^s))$ \cite[Sec.1]{HN}.

Let $\lambda_1\in\Q^-$ be a negative rational number, $\lambda_1=-h_1/e_1$, with $h_1,e_1$ po\-sitive coprime integers. The triple $(\phi_1,v_1,\lambda_1)$ determines a discrete valuation $v_2$ on $\qp(x)$, constructed as follows: for any  nonzero polynomial $g(x)\in\zpx$, take a line of slope $\lambda_1$ far below $N_1(g)$ and let it shift upwards till it touches the polygon for the first time; if $H$ is the ordinate at the origin of this line, then $v_2(g(x))=e_1 H$, by definition. Also, the triple $(\phi_1,v_1,\lambda_1)$ determines a residual polynomial  operator
$$
R_1:=R_{\phi_1,v_1,\lambda_1}\colon \zpx\lra \ff1[y],\quad \ff1:=\ff0[y]/(\psi_0(y)),
$$
which is a kind of reduction of first order of $g(x)$ \cite[Def.1.9]{HN}.

Let $\psi_1(y)\in\ff1[y]$ be a monic irreducible polynomial, $\psi_1(y)\ne y$. The triple $\ty=(\phi_1(x);\lambda_1,\psi_1(y))$ is called a \emph{type of order one}. Given any such type, one can compute a representative of $\ty$; that is, a monic polynomial $\phi_2(x)\in\Z[x]$ of degree $e_1\deg\psi_1\deg\phi_1$, satisfying
$R_1(\phi_2)(y)\sim\psi_1(y)$. Now we may start over with the pair $(\phi_2,v_2)$ and repeat all constructions in order two.

The iteration of this procedure leads to the concept of \emph{type of order $r$} \cite[Sec.2]{HN}. A type of order $r\ge 1$ is a chain:
$$
\ty=(\phi_1(x);\lambda_1,\phi_2(x);\cdots;\lambda_{r-1},\phi_r(x);\lambda_r,\psi_r(y)),
$$
where $\phi_1(x),\dots,\phi_r(x)$ are monic polynomials in $\Z[x]$ that are irreducible in $\zpx$, $\lambda_1,\dots,\lambda_r$ are negative rational numbers, and $\psi_r(y)$ is a polynomial over certain finite field $\ff{r}$ (to be specified below), that satisfy the following recursive properties:

\begin{enumerate}
\item $\phi_1(x)$ is irreducible modulo $p$.  We define $\psi_0(y):=R_0(\phi_1)(y)\in \ff0[y]$, $\ff1=\ff0[y]/(\psi_0(y))$.
\item For all $1\le i<r$, $N_i(\phi_{i+1}):=N_{\phi_i,v_i}(\phi_{i+1})$ is one-sided of slope $\lambda_i$, and $R_i(\phi_{i+1})(y):=R_{\phi_i,v_i,\lambda_i}(\phi_{i+1})(y)\sim \psi_i(y)$, for some monic irreducible polynomial $\psi_i(y)\in \ff{i}[y]$. We define $\ff{i+1}=\ff{i}[y]/(\psi_i(y))$.
\item $\psi_r(y)\in\ff{r}[y]$ is a monic irreducible polynomial, $\psi_r(y)\ne y$.
\end{enumerate}

Thus, a type of order $r$ is an object structured in $r$ levels. In the computational representation of a type, several invariants are stored at each level, $1\le i\le r$. The most important ones are:
$$
\begin{array}{ll}
\phi_i(x), & \mbox{monic polynomial in }\Z[x], \mbox{  irreducible over }\zpx\\
m_i, & \deg \phi_i(x),\\
v_i(\phi_i)&\mbox{non-negative integer},\\
\lambda_i=-h_i/e_i,& \mbox{$h_i,e_i$ positive coprime integers},\\
\ell_i,\ell'_i,& \mbox{a pair of integers satisfying }\ell_ih_i-\ell'_ie_i=1,\\
\psi_i(y), & \mbox{monic irreducible polynomial in }\ff{i}[y],\\
f_i& \deg \psi_i(y),\\
z_i& \mbox{the class of $y$ in $\ff{i+1}$, so that }\psi_i(z_i)=0.
\end{array}
$$
Take $f_0:=\deg \psi_0$, and let $z_0\in\ff{1}$ be the class of $y$, so that $\psi_0(z_0)=0$. Note that $m_i=(f_0f_1\cdots f_{i-1})(e_1\cdots e_{i-1})$, $\ff{i+1}=\ff{i}[z_i]$, and $\dim_{\ff0}\ff{i+1}=f_0f_1\cdots f_i$. The discrete valuations $v_1,\dots,v_{r+1}$ on the field $\qp(x)$ are essential invariants of the type.

\begin{definition}\label{defs}
 Let $g(x)\in\zpx$ be a monic separable polynomial, and $\ty$ a type of order $r\ge 1$.

(1) \ We say that $\ty$ \emph{divides} $g(x)$ (and we write $\ty \,|\, g(x)$), if $\psi_r(y)$ divides $R_r(g)(y)$ in $\ff{r}[y]$.

(2) \ We say that $\ty$ is \emph{$g$-complete} if $\operatorname{ord}_{\psi_r}(R_r(g))=1$. In this case, $\ty$ singles out a monic irreducible factor $g_\ty(x)\in\zpx$ of $g(x)$, uniquely determined by the property  $R_r(g_\ty)(y)\sim\psi_r(y)$. If $K_\ty$ is the extension of $\qp$ determined by $g_\ty(x)$, then $$e(K_\ty/\qp)=e_1\cdots e_r,\qquad f(K_\ty/\qp)=f_0f_1\cdots f_r.$$

(3) \ A \emph{representative} of $\ty$ is a monic polynomial $\phi_{r+1}(x)\in\Z[x]$, of  degree $m_{r+1}=e_rf_rm_r$ such that $R_r(\phi_{r+1})(y)\sim \psi_r(y)$. This polynomial is necessarily irreducible in $\zpx$. By the definition of a type, each $\phi_{i+1}(x)$ is a representative of the truncated type of order $i$
$$
\op{Trunc}_{i}(\ty):=(\phi_1(x);\lambda_1,\phi_2(x);\cdots;\lambda_{i-1},\phi_i(x);\lambda_i,\psi_i(y)).
$$

(4) \ We say that $\ty$ is \emph{optimal} if $m_1<\cdots<m_r$, or equivalently, if $e_if_i>1$, for all $1\le i< r$.
\end{definition}

\begin{lemma}\label{vjphii}
Let $\ty$ be a type of order $r$. Then, $v_j(\phi_i)=(m_i/m_j)v_j(\phi_j)$, for all $j<i\le r$.
\end{lemma}

\begin{proof}
Let $\phi_i(x)=\sum_{s\ge 0} a_s(x)\phi_j(x)^s$ be the $\phi_j$-adic development of $\phi_i$. By \cite[Lem.2.17]{HN}, $v_j(\phi_i)=\min_{s\ge 0}\{v_j(a_s\phi_j^s)\}$. Now, $N_j(\phi_i)$ is one-sided of slope $\lambda_j$, because $\phi_i$ is a polynomial of type $\op{Trunc}_j(\ty)$ \cite[Def.2.1+Lem.2.4]{HN}. Since the principal term of the development is $\phi_j^{m_i/m_j}$, we get
$v_j(\phi_i)=v_j(\phi_j^{m_i/m_j})=(m_i/m_j)v_j(\phi_j)$.
\end{proof}

\subsection{Certain rational functions}\label{subsecratfunctions}
Let $\ty$ be a type of order $r$. We attach to $\ty$ seve\-ral rational functions in $\Q(x)$ \cite[Sec.2.4]{HN}. Note that $v_i(\phi_i)$ is always divisible by $e_{i-1}$
\cite[Thm.2.11]{HN}.

\begin{definition}\label{ratfracs}
Let $\,\pi_0(x)=1$, $\pi_1(x)=p$. We define recursively for all $1\le i\le r$:
$$
\Phi_i(x)=\dfrac{\phi_i(x)}{\pi_{i-1}(x)^{v_i(\phi_i)/e_{i-1}}},\qquad
\ga_i(x)=\dfrac{\Phi_i(x)^{e_i}}{\pi_i(x)^{h_i}},\qquad
\pi_{i+1}(x)=\dfrac{\Phi_i(x)^{\ell_i}}{\pi_i(x)^{\ell'_i}}.
$$
These rational functions can be written as a product of powers of $p,\phi_1(x),\dots,\phi_r(x)$, with integer exponents.
\end{definition}

\noindent{\bf Notation. }Let $\Psi(x)=p^{n_0}\phi_1(x)^{n_1}\cdots\phi_s(x)^{n_s}\in \Q(x)$ be a rational function which is a product of powers of $p,\phi_1,\dots,\phi_s$, with integer exponents. We denote:
$$
\log \Psi=(n_0,\dots,n_s)\in \Z^{s+1}.
$$

The next result is inspired in \cite[Cor.4.26]{HN}.

\begin{lemma}\label{prodgammas}
Let $F(x)\in\Z_p[x]$ be a monic irreducible polynomial divisible by $\ty$, and let $\alpha\in\qpb$ be a root of $F(x)$. For some $1\le s\le r$, let $\Psi(x)=p^{n_0}\phi_1(x)^{n_1}\cdots\phi_s(x)^{n_s}$ be a rational function in $\Q(x)$, such that $v(\Psi(\alpha))=0$. Then,
$$\Psi(x)=\gamma_1(x)^{t_1}\cdots\gamma_s(x)^{t_s},
$$
for certain integer exponents $t_1,\dots,t_s\in\Z$, which can be computed by the following recursive procedure:

{\tt

\qquad vector=$(n_0,\dots,n_s)$

\nopagebreak
\qquad for i=s to 1 by -1 do

\nopagebreak
\qquad\qquad$t_i$=vector[i]/$e_i$

\nopagebreak
\qquad\qquad vector=vector-$t_i\log\gamma_i$

\qquad end for}

\end{lemma}

\begin{proof}
By \cite[(17)]{HN} and \cite[Cor.3.2]{HN}:
$$\log\Phi_s=(\dots\dots,1),\ \log\pi_s=(\dots\dots,0),\ \log \gamma_s=(\dots\dots,e_s)\in\Z^{s+1},$$
$$
v(\phi_s(\alpha))=\sum_{i=1}^s e_if_i\cdots e_{s-1}f_{s-1}\dfrac{h_i}{e_1\cdots e_i},
\qquad v(\gamma_s(\alpha))=0.
$$
Since $v(\Psi(\alpha))=0$, the formula for $v(\phi_s(\alpha))$ shows that $e_s| n_s$. Thus, we can replace $\Psi(x)$ by $\Psi(x)\gamma_s^{-n_s/e_s}$ and iterate the argument. Since $v(\gamma_s(\alpha))=0$, the new $\Psi(x)$ satisfies $v(\Psi(\alpha))=0$ as well, and the $s$-th coordinate of $\log\Psi$ is zero. At the last step ($s=1$), we get $\Psi(x)=p^{n'_0}\phi_1^{n'_1}$, with $n'_0+n'_1(h_1/e_1)=0$.
Then, clearly $\Psi(x)=\gamma_1(x)^{n'_1/e_1}$.
\end{proof}

\subsection{Montes algorithm and the secondary invariants}
Let $f(x)\in\Z[x]$ be a monic irreducible polynomial. At the input of the pair $(f(x),p)$, Montes algorithm computes a family $\ty_1,\dots,\ty_s$ of $f$-complete and optimal types in one-to-one correspondence with the irreducible factors $f_{\ty_1}(x),\dots,f_{\ty_s}(x)$ of $f(x)$ in $\zpx$. This one-to-one correspondence is determined by:
\begin{enumerate}
 \item For all $1\le i\le s$, the type $\ty_i$ is $f_{\ty_i}$-complete.
\item For all $j\ne i$, the type $\ty_j$ does not divide $f_{\ty_i}(x)$.
\end{enumerate}

The algorithm starts by computing the order zero types determined by the irreducible factors of $f(x)$ modu\-lo $p$, and then proceeds to enlarge them in a convenient way till the whole list of $f$-complete optimal types is obtained \cite{GMNalgorithm}.

With regard to the computation of generators of the prime ideals and chinese remainder multipliers, the algorithm is slightly modified to compute and store some other (secondary) invariants at each level of all types $\ty$ considered by the algorithm:
$$\begin{array}{ll}
\op{Refinements}_i, & \mbox{ a list of pairs $[\phi(x),\lambda]$, where $\phi$ is a representative of}\\& \qquad\op{Trunc}_{i-1}(\ty)\mbox{  and $\lambda$ a negative slope},\\
u_i, & \mbox{ a nonnegative integer called the \emph{height}},\\
\op{Quot}_i,& \mbox{ a list of $e_i$ polynomials in }\Z[x], \\
\log \Phi_i,& \mbox{ a vector }(n_0,\dots,n_i)\in\Z^{i+1},\\
\log \pi_i,& \mbox{ a vector }(n_0,\dots,n_{i-1},0)\in\Z^{i+1},\\
\log \gamma_i,& \mbox{ a vector }(n_0,\dots,n_i)\in\Z^{i+1}.
\end{array}
$$

Let us briefly explain the flow of the algorithm and the computation of these invariants. Suppose a type of order $i-1$ dividing $f(x)$ is considered,
$$\ty=(\phi_1(x);\lambda_1,\phi_2(x);\cdots;\lambda_{i-2},\phi_{i-1};\lambda_{i-1},\psi_{i-1}(y)).$$
A representative $\phi_i(x)$ is constructed. Suppose that either $i=1$ or $m_1<\cdots<m_i$. Let $\ell=\ord_{\psi_{i-1}}R_{i-1}(f)$. If $\ty$ is not $f$-complete ($\ell>1$), it may ramify to produce new types, that will be germs of distinct $f$-complete types. To carry out this ramification process we compute simultaneously the first $\ell+1$ coefficients of the $\phi_i$-adic development of $f(x)$ and the corresponding quotients:
\begin{equation}\label{quotients}
\begin{array}{rcl}
f(x)&=&\phi_i(x)q_1(x)+a_0(x),\\
q_1(x)&=&\phi_i(x)q_2(x)+a_1(x),\\
\cdots&&\cdots\\
q_\ell(x)&=&\phi_i(x)q_{\ell+1}(x)+a_\ell(x).
\end{array}
\end{equation}
The Newton polygon of $i$-th order of $f(x)$, $N_i(f)$, is the lower convex envelope of the set of points  $(s,v_i(a_s\phi_i^s))$ of the plane, for all $s\ge 0$. However, we need to build up only the \emph{principal part} of this polygon, $N^-_i(f)=N^-_{\phi_i,v_i}(f)$, formed by the sides of negative slope of $N_i(f)$. By  \cite[Lem.2.17]{HN}, this latter polygon is the lower convex envelope of the set of points  $(s,v_i(a_s\phi_i^s))$ of the plane, for $0\le s\le \ell$. For each side of slope (say) $\lambda$ of $N^-_i(f)$, the residual polynomial
$R_{\lambda}(f)(y)=R_{\phi_i,v_i,\lambda}(f)(y)\in\ff{i}[y]$ is computed and factorized into a product of irreducible factors. The type $\ty$ branches in principle into as many types as pairs $(\lambda,\psi(y))$, where $\lambda$ runs on the negative slopes of $N^-_i(f)$ and $\psi(y)$ runs on the different irreducible factors of $R_{\lambda}(f)(y)$. If one of these branches
$$
\ty_{\lambda,\psi}:=(\phi_1(x);\lambda_1,\phi_2(x);\cdots;\lambda_{i-1},\phi_i(x);\lambda,\psi(y)),
$$ is $f$-complete, we store this type in an specific list and we go on with the analysis of other branches. Otherwise, we compute a representative $\phi_{\lambda,\psi}(x)$ of $\ty_{\lambda,\psi}$. Let $e$ be the least non-negative denominator of $\lambda$ and $f=\deg \psi$. Then we proceed in a different way according to $ef=1$ or $ef>1$.

If $ef>1$, then $\deg \phi_{\lambda,\psi}>m_i$, so that $\phi_{i+1}:=\phi_{\lambda,\psi}$ may be used to enlarge $\ty_{\lambda,\psi}$ into several optimal types of order $i+1$. We store the invariants
$$
\begin{array}{l}
\phi_i, \ m_i=\deg \phi_i, \ v_i(\phi_i),\ \lambda_i=\lambda, \ h_i,\ e_i=e,\ \ell_i,\ \ell'_i,\ \psi_i=\psi, \ f_i=f,\ z_i, \\ u_i,\ \op{Quot}_i,\ \log\Phi_i,\ \log\pi_i,\ \log\gamma_i
\end{array}
$$
at the $i$-th level of  $\ty_{\lambda,\psi}$, and then we proceed to enlarge the type.
The invariant $v_i(\phi_i)$ is recursively computed by using \cite[Prop.2.7+Thm.2.11]{HN}:
$$
v_i(\phi_i)=\left\{\begin{array}{ll}
0,&\mbox{ if }i=1,\\
e_{i-1}f_{i-1}(e_{i-1}v_{i-1}(\phi_{i-1})+h_{i-1}),&\mbox{ if }i>1.
\end{array}
\right.
$$
Let $s_i$ be the abscissa of the right end point of the side of $N^-_i(f)$ of slope $\lambda$; the secondary invariants are computed as:
\begin{equation}\label{secondaryinvariants}
\as{1.2}
\begin{array}{l}
u_i=v_i(a_{s_i}),\\
\op{Quot}_i=[1,\,q_{s_i-1}(x),\,\dots,\,q_{s_i-e+1}(x)],\\
\log \Phi_i=(n_0,\dots,n_{i-1},1), \mbox{ where }(n_0,\dots,n_{i-1})=-(v_i(\phi_i)/e_{i-1})\log \pi_{i-1},\\
\log \pi_i=\ell_{i-1}\log\Phi_{i-1}-\ell'_{i-1}\log \pi_{i-1},\\
\log \gamma_i=e_i\log\Phi_i-h_i\log \pi_i.
\end{array}
\end{equation}
For $i=1$ we take $\log \Phi_1=(0,1)$ and $\log \pi_1=(1,0)$.
Note that all these secondary invariants depend only on $\lambda$ and not on $\psi$. They are computed only once for each side of $N^-_i(f)$ and then stored in the different optimal branches that share the same slope.
The type $\ty_{\lambda,\psi}$ of order $i$ is then ready for further analysis.

If $ef=1$ then $\deg \phi_{\lambda,\psi}=m_i$, so that the enlargements of $\ty_{\lambda,\psi}$ that would result from building up an $(i+1)$-th level from $\phi_{\lambda,\psi}$ would not be optimal. In this case, we replace $\ty_{\lambda,\psi}$ by  all types
$$
\ty'_{\lambda',\psi'}:=(\phi_1(x);\lambda_1,\phi_2(x);\cdots;\lambda_{i-1},\phi'_i(x);\lambda',\psi'(y)),
$$
obtained by enlarging $\ty$ with $i$-th levels deduced from the consideration of $\phi_{\lambda,\psi}$ as a new (and better) representative of $\ty$: $\phi_i'(x):=\phi_{\lambda,\psi}(x)$, but taking into account only the slopes $\lambda'$ of $N_{\phi'_i,v_i}(f)$ satisfying $\lambda'<\lambda$; this is called a \emph{refinement step}  \cite[Sect.3.2]{GMNalgorithm}. If the total number of pairs $(\lambda,\psi)$ is greater than one, we append to the list $\op{Refinements}_i$ of all types  $\ty'_{\lambda',\psi'}$ the pair $[\phi_i(x),\lambda]$.
Some of these new branches $\ty'_{\lambda',\psi'}$ of order $i$ may be $f$-complete, some may lead to optimal enlargements and some may lead to further refinement at the $i$-th level. Thus, in general, the list $\op{Refinements}_i$ of a type $\ty$ is an ordered sequence of pairs:
$$
\op{Refinements}_i=[[\phi_i^{(1)},\lambda_i^{(1)}], \cdots ,[\phi_i^{(s)},\lambda_i^{(s)}]],
$$
reflecting the fact that along the construction of $\ty$, $s$ or more successive refinement steps occurred at the $i$-th level. All polynomials $\phi_i^{(k)}$ are representatives of $\op{Trunc}_{i-1}(\ty)$, and the slopes $\lambda_i^{(k)}$ grow strictly in absolute size: $|\lambda_i^{(1)}|<\cdots <|\lambda_i^{(s)}|$.  We recall that we store only the pairs  $[\phi_i^{(k)},\lambda_i^{(k)}]$ corresponding to a refinement step that occurred simultaneously with some branching. For instance, if $N_{\phi_i^{(k)},v_i}^-(f)$ has only one side with integer slope $\lambda_i^{(k)}\in\Z$ (i.e. $e=1)$ and the corresponding residual polynomial is the power of an irreducible polynomial of degree one $(f=1)$, then the pair $[\phi_i^{(k)},\lambda_i^{(k)}]$ is not included in the list  $\op{Refinements}_i$.

After a finite number of branching, enlargement and/or refinement steps, all types become $f$-complete and optimal. If $\ty$ is an $f$-complete and optimal type of order $r$, then the \emph{Okutsu depth} of $f_\ty(x)$ is
\cite[Thm.4.2]{GMNokutsu}:
\begin{equation}\label{depth}
 R=\left\{\begin{array}{ll}
 r,&\mbox{ if }e_rf_r>1, \\r-1,&\mbox{ if }e_rf_r=1.
\end{array}\right.
\end{equation}

The invariants $v_{i+1},\,h_i,\,e_i,\,f_i$ at each level $1\le i\le R$ are canonical (depend only on $f(x)$) \cite[Cor.3.7]{GMNokutsu}. On the other hand, the polynomials $\phi_i(x),\psi_i(y)$ depend on several choices, some of them caused by the lifting of elements of a finite field to rings of characteristic zero. However, the sequence $[\phi_1,\dots,\phi_R]$ is an \emph{Okutsu frame} of $f_\ty(x)$ \cite[Sec.2]{GMNokutsu}; in the original terminology of Okutsu, the polynomials $\phi_1,\dots,\phi_R$ are \emph{primitive divisor polynomials} of $f_\ty(x)$ \cite{Ok}.

\subsection{Montes approximations to the irreducible $p$-adic factors}\label{subsecApprox}
Once an $f$-complete and optimal type $\ty$ of order $r$ is computed, Montes algorithm attaches to it an $(r+1)$-level that carries only the invariants:
\begin{equation}\label{lastlevel}\as{1.2}
\begin{array}{l}
\phi_{r+1},\ m_{r+1},\ v_{r+1}(\phi_{r+1}),\ \la_{r+1}=-h_{r+1},\ e_{r+1}=1,\ \psi_{r+1},\ f_{r+1}=1,\\
z_{r+1},\ \log\Phi_{r+1},\ \log\pi_{r+1},\ \log\gamma_{r+1}.
\end{array}
\end{equation}
The polynomial $\phi_{r+1}(x)$ is a representative of $\ty$. The invariants $\lambda_{r+1}$, $\psi_{r+1}$ are deduced from the computation of the principal part of the Newton polygon of $(r+1)$-th order of $f(x)$ (which is a single side of length one) and the corresponding resi\-dual polynomial  $R_{r+1}(f)(y)\in\ff{r+1}[y]$ (which has degree one).

If $\p$ is the prime ideal corresponding to $\ty$, we denote
$$\as{1.3}
\begin{array}{l}
f_\p(x):=f_\ty(x)\in\zpx,\quad \phi_\p(x):=\phi_{r+1}(x)\in\Z[x],\quad \ff{\p}:=\ff{r+1},\\
\ty_\p:=(\phi_1;\lambda_1,\phi_2;\cdots,\phi_r;\lambda_r,\phi_{\p};\lambda_{r+1},\psi_{r+1}),
\end{array}
$$
and we say that $\p=[p;\phi_1,\dots,\phi_r,\phi_\p]$ is the \emph{Okutsu-Montes representation} of $\p$.

Note that $\ty_\p$ is still an $f$-complete type of order $r+1$, but it may eventually be non-optimal because $m_{r+1}=m_r$, if $e_rf_r=1$.

The polynomial $\phi_\p(x)$ is a \emph{Montes approximation} to the $p$-adic irreducible factor $f_\p(x)$ \cite[Sec.4.1]{GMNokutsu}. Several arithmetic tasks involving prime ideals (tasks (1), (3), (5), (6) and (7) from the list given in the Introduction) require the computation of a Montes approximation with a sufficiently large value of the last slope $|\lambda_{r+1}|=h_{r+1}$. This can be achieved by applying a finite number of refinement steps at the $(r+1)$-th level, to the type $\ty_\p$ of order $r+1$, as described in \cite[Sec.4.3]{GMNokutsu}. This procedure has a linear convergence. In \cite{GNP} a more efficient \emph{single-factor lift} algorithm is developed, which is able to improve the Montes approximations to $f_\p(x)$ with quadratic convergence.

\section{$\p$-adic valuation and factorization}\label{secPadic}
In section \ref{subsecPadic} we compute the $\p$-adic valuation, $v_\p\colon K^*\to \Z$, determined by a prime ideal $\p$, in terms of the data contained in the Okutsu-Montes representation of $\p$. In section \ref{subsecFactorization} we describe a procedure to find the prime ideal decomposition of any fractional ideal of $K$. From the computational point of view, this procedure is based on three ingredients:
\begin{enumerate}
\item The factorization of integers.
\item Montes algorithm to find the prime ideal decomposition of a prime number.
\item The computation of $v_\p$ for some prime ideals $\p$.
\end{enumerate}
The routines (2) and (3) run extremely fast in practice (see section \ref{secKOM}).

It is well-known how to add, multiply and intersect fractional ideals once their prime ideal factorization is available. We omit the description of the routines that carry out these tasks.

From now on, to any prime ideal $\p$ of $K$ we attach the data $\ty_\p$, $f_\p(x)$, $\phi_\p(x)$, $\ff{\p}$,  described in section \ref{subsecApprox}. Also, we choose a root
$\t_\p\in\qpb$ of $f_\p(x)$, we consider the local field $K_\p=\qp(\t_\p)$, and we denote by $\Z_{K_\p}$ the ring of integers of $K_\p$.

\subsection{Computation of the $\p$-adic valuation}\label{subsecPadic}
Let $p$ be a prime number and $v\colon \qpb^{\,*}\to \Q$, the canonical $p$-adic valuation. Let $\p$ be a prime ideal of $K$ lying above $p$, corresponding to an $f$-complete type $\ty_\p$ with an added $(r+1)$-th level, as indicated in section \ref{subsecApprox}.
We shall freely use all invariants of $\ty_\p$ described in section \ref{secMontes}. By item 2 of Definition \ref{defs} we know that
$$
e(\p/p)=e_1\cdots e_r,\qquad f(\p/p)=f_0f_1\cdots f_r.
$$

The residue field $\Z_{K_\p}/\p\Z_{K_\p}$ can be identified to the finite field $$\ff{\p}:=\ff{r+1}=\ff{0}[z_0,z_1,\dots,z_r].$$
More precisely, in \cite[(27)]{HN} we construct an explicit isomorphism
\begin{equation}\label{embedding}
\gamma\colon \ff{\p}\iso \Z_{K_\p}/\p\Z_{K_\p}, \qquad z_0\mapsto \tb, \ z_1\mapsto \gb1,\ \dots,\ z_r\mapsto \gb{r},
\end{equation}
where we indicate by a bar the canonical reduction map, $\Z_{K_\p}\lra \Z_{K_\p}/\p\Z_{K_\p}$. We denote by $\op{lred}_\p \colon\Z_{K_\p}\lra \ff{\p}$,  the reduction map obtained by composition of the canonical reduction map with the inverse of the isomorphism (\ref{embedding}).
\begin{equation}\label{lred}
\op{lred}_\p\colon \Z_{K_\p}\lra \Z_{K_\p}/\p\Z_{K_\p} \stackrel{\gamma^{-1}}\lra \ff{\p}.
\end{equation}

Consider the topological embedding $\iota_\p\colon K\hookrightarrow K_\p$ determined by sending $\t$ to $\t_\p$. We have: $v_\p(\alpha)=e(\p/p)v(\iota_\p(\alpha))$, for all $\alpha\in K$. In particular, for any polynomial $g(x)\in\Z[x]$,
\begin{equation}\label{tautology}
v_\p(g(\t))=e(\p/p)v(g(\t_\p)).
\end{equation}

Any $\alpha\in K^*$ can be expressed as $\alpha=(a/b)g(\t)$, for some coprime positive integers $a,b$ and some primitive polynomial $g(x)\in\Z[x]$. By (\ref{tautology}), $$v_\p(\alpha)=e(\p/p)(v(g(\t_\p))+v(a/b)).$$ Thus, it is sufficient to learn to compute $v(g(\t_\p))$. The condition $v(g(\t_\p))=0$ is easy to check \cite[Lem.2.2]{GMNokutsu}:
\begin{equation}\label{v=0}
v(g(\t_\p))=0 \ \sii\ \psi_0\nmid R_0(g).
\end{equation}
If $\psi_0\mid R_0(g)$, the computation of $v(g(\t_\p))$ can be based on the following proposition, which is easily deduced from \cite[Prop.3.5]{HN} and \cite[Cor.3.2]{HN}.

\begin{proposition}\label{vgt}Let $\p,\,f_\p(x),\,\t_\p$ be as above. Let $\ty$ be a type of order $R$ dividing $f_\p(x)$,
and let $g(x)\in\Z[x]$ be a nonzero polynomial. For any $1\le i\le R$, take a line $L_{\lambda_i}$ of slope $\lambda_i$ far below $N_i(g)$, and let it shift upwards till it touches the polygon for the first time. Let $S$ be the intersection of this line with $N_i(g)$, let $(s,u)$ be the coordinates of the left end point of $S$, and let $H=u+s|\lambda_i|$ be the ordinate at the origin of this line. Then,
\begin{enumerate}
\item $v(g(\t_\p))\ge H/e_1\cdots e_{i-1}$, and equality holds if and only if $\op{Trunc}_i(\ty)\nmid g(x)$.
\item If equality holds, then $v(g(\t_\p))=v(\Phi_i(\t_\p)^s\pi_i(\t_\p)^u)$ and
$$
\op{lred}_\p\left(\dfrac{g(\t_\p)}{\Phi_i(\t_\p)^s\pi_i(\t_\p)^u}\right)=R_i(g)(z_i)\ne0.
$$\qed
\end{enumerate}
\end{proposition}

Figure 1 shows that the segment $S$ may eventually be reduced to a point. In this case, the residual polynomial $R_i(g)(y)$ is a constant \cite[Def.2.21]{HN}, so that $\op{Trunc}_i(\ty)\nmid g(x)$ automatically holds.
\begin{center}
\setlength{\unitlength}{5.mm}
\begin{picture}(16,6)
\put(2.85,1.85){$\bullet$}\put(1.85,2.85){$\bullet$}
\put(-1,0){\line(1,0){7}}\put(0,-1){\line(0,1){6}}
\put(3,2){\line(-1,1){1}}\put(3.02,2){\line(-1,1){1}}
\put(3,2){\line(3,-1){1}}\put(3.02,2){\line(3,-1){1}}
\put(2,3){\line(-1,2){1}}\put(2.02,3){\line(-1,2){1}}
\put(6,.5){\line(-2,1){7}}
\put(5.2,1){\begin{footnotesize}$L_{\lambda_i}$\end{footnotesize}}
\multiput(3,-.1)(0,.25){9}{\vrule height2pt}
\multiput(-.1,2)(.25,0){12}{\hbox to 2pt{\hrulefill }}
\put(1.8,4.5){\begin{footnotesize}$N_i(g)$\end{footnotesize}}
\put(-.6,1.85){\begin{footnotesize}$u$\end{footnotesize}}
\put(-.6,3.1){\begin{footnotesize}$H$\end{footnotesize}}
\put(2.9,2.3){\begin{footnotesize}$S$\end{footnotesize}}
\put(2.8,-.6){\begin{footnotesize}$s$\end{footnotesize}}
\put(14.8,1.35){$\bullet$}\put(12.85,2.35){$\bullet$}
\put(9,0){\line(1,0){8}}\put(10,-1){\line(0,1){6}}
\put(15,1.5){\line(-2,1){2}}\put(15.02,1.5){\line(-2,1){2}}
\put(15,1.5){\line(3,-1){1}}\put(15.02,1.5){\line(3,-1){1}}
\put(13,2.5){\line(-1,2){1.3}}\put(13.02,2.5){\line(-1,2){1.3}}
\put(17,.5){\line(-2,1){8}}
\put(16.6,.8){\begin{footnotesize}$L_{\lambda_i}$\end{footnotesize}}
\put(14,2.1){\begin{footnotesize}$S$\end{footnotesize}}
\put(12.85,-.6){\begin{footnotesize}$s$\end{footnotesize}}
\multiput(13,-.1)(0,.25){11}{\vrule height2pt}
\multiput(9.9,2.55)(.25,0){12}{\hbox to 2pt{\hrulefill }}
\put(12.5,4.5){\begin{footnotesize}$N_i(g)$\end{footnotesize}}
\put(9.5,2.35){\begin{footnotesize}$u$\end{footnotesize}}
\put(9.4,3.6){\begin{footnotesize}$H$\end{footnotesize}}
\end{picture}
\end{center}\be
\begin{center}
Figure 1
\end{center}\be

We may compute $v_\p(g(\t))=e(\p/p)v(g(\t_\p))$ by applying Proposition \ref{vgt} to the type $\ty_\p$. If for some $1\le i\le r+1$, the truncated type $\op{Trunc}_i(\ty_\p)$ does not divide $g(x)$, we compute $v(g(\t_\p))$ as indicated in item 1 of this proposition. Nevertheless, it may occur that $\op{Trunc}_i(\ty_\p)$ divides $g(x)$ for all $1\le i\le r+1$ (for instance, if $g(x)$ is a multiple of $\phi_\p(x)=\phi_{r+1}(x)$). In this case, we compute an improvement of the Montes approximation $\phi_\p(x)$ by applying the single-factor lift routine \cite{GNP}; then, we replace the $(r+1)$-th level of $\ty_\p$ by the invariants (\ref{lastlevel}) determined by the new choice of $\phi_{r+1}(x)=\phi_\p(x)$, and we test again if $\ty_\p=\op{Trunc}_{r+1}(\ty_\p)$ divides $g(x)$.

If $\ty_\p$ divides $g(x)$, then $\phi_\p(x)$ is simultaneously close to a $p$-adic irreducible factor of $f(x)$ and to a $p$-adic irreducible factor of $g(x)$; hence, if $f(x)$ and $g(x)$ do not have a common $p$-adic irreducible factor, after a finite number of steps the renewed type $\ty_\p$ will not divide $g(x)$. On the other hand, if $f(x)$ and $g(x)$ have a common $p$-adic irreducible factor, they must have a common irreducible factor in $\Z[x]$ too; since $f(x)$ is irreducible, necessarily $f(x)$ divides $g(x)$ and $g(\t)=0$.

We may summarize the routine to compute $v_\p(\alpha)$ as follows.     \medskip

\nn{\bf Input: } $\alpha\in K^*$ and a prime ideal $\p$ determined by a type $\ty_\p$ of order $r+1$.

\nopagebreak
\nn{\bf Output: } $v_\p(\alpha)$.\medskip

\nopagebreak
\nn{\bf1. }Write $\alpha=\frac ab g(\t)$, with $a,b$ coprime integers and $g(x)\in\Z[x]$ primitive.

\nopagebreak
\nn{\bf2. }Compute $\nu=v(a/b)$.

\nn{\bf3. }if $\psi_0\nmid R_0(g)$ then return $v_\p(\alpha)=e(\p/p)\nu$.

\nn{\bf4. }for $i=1$ to $r+1$ do

\qquad compute $N_i^-(g)$, $R_i(g)$, and the ordinate $H$ of Proposition \ref{vgt}.

\qquad if $\psi_i\nmid R_i(g)$ then return $v_\p(\alpha)=e(\p/p)((H/e_1\cdots e_{i-1})+\nu)$.

\nn\hphantom{\bf4. }end for.

\nn{\bf5. }while $\psi_{r+1}\mid R_{r+1}(g)$ do

\qquad improve $\phi_\p$ and compute the new values $\la_{r+1}$, $\psi_{r+1}$.

\qquad compute $N_{r+1}^-(g)$, $R_{r+1}(g)$, and the ordinate $H$ of Proposition \ref{vgt}.

\nn\hphantom{\bf5. }end while.

\nn{\bf6. }return $v_\p(\alpha)=e(\p/p)((H/e_1\cdots e_r)+\nu)$.

\subsection{Factorization of fractional ideals}\label{subsecFactorization}
For any $\alpha\in K^*$, the factorization of the principal ideal generated by $\alpha$ is
$$\alpha\Z_K=\prod_\p\p^{v_\p(\alpha)}.$$
Let $\alpha=(a/b)g(\t)$, for some positive coprime integers $a,b$ and some primitive polynomial $g(x)\in\Z[x]$. Then, $v_\p(\alpha)=0$ for all prime ideals $\p$ whose underlying prime number $p$ does not divide the product $ab\op{N}_{K/\Q}(g(\t))=ab\op{Resultant}(f,g)$.

Also, if $\p$ is a prime ideal of $K$ and $\a$, $\b$ are fractional ideals, we have
$$
v_\p(\a+\b)=\min\{v_\p(\a),v_\p(\b)\}.
$$
Thus, the $\p$-adic valuation of the fractional ideal $\a$ generated by $\alpha_1,\dots,\alpha_m\in K^*$ is:
$v_\p(\a)=\min_{1\le i\le m}\{v_\p(\alpha_i)\}$.

After these considerations, it is straightforward to deduce a factorization routine of fractional ideals from the   routines computing prime ideal decompositions of prime numbers and $\p$-valuations of elements of $K^*$ with respect to prime ideals $\p$.\medskip

\nn{\bf Input: } a family $\alpha_1,\dots,\alpha_m\in K^*$ of generators of a fractional ideal $\a$.

\nn{\bf Output: } the prime ideal decomposition $\a=\prod_{\p}\p^{a_\p}$.\medskip

\nn{\bf1. }For each $1\le i\le m$, write $\alpha_i=(a_i/b_i) g_i(\t)$, with $a_i,b_i$ coprime integers and $g_i(x)\in\Z[x]$ primitive; then compute $N_i=\op{N}_{K/\Q}(g_i(\t))$.

\nn{\bf2. }Compute $N=\op{gcd}(a_1N_1,\dots,a_mN_m)$ and $M=\op{lcm}(b_1,\dots,b_m)$.

\nn{\bf3. }Factorize $N$ and $M$ and store all their prime factors in a list $\pp$.

\nn{\bf4. }For each $p\in\pp$ apply Montes algorithm to obtain the prime ideal decomposition of $p$, and
 for each $\p|p$, take $a_\p=\min_{1\le i\le m}\{v_\p(\alpha_i)\}$.

\nn{\bf5. }Return the list of pairs $[\p,a_\p]$ for all $\p$ with $a_\p\ne0$. \medskip

The bottleneck of this routine is step 3. We get a fast facto\-rization routine in the number field $K$, as long as the integers $N$, $M$ attached to the ideal $\a$ may be easily factorized.

\section{Computation of generators}\label{secGenerators}
In \cite[Sec.4]{GMNalgorithm} we gave an algorithm to compute generators of the prime ideals as certain rational functions of the $\phi$-polynomials. Some inversions in $K$, one for each prime ideal, were needed. These inversions dominated the complexity of the algorithm, and they were a bottleneck that prevented the computation of generators for number fields of large degree.

In sections \ref{subsectPseudo} and \ref{subsectgenerators} we construct a two-element representation of prime ideals, which does not need any inversion in $K$. As a consequence, this construction works extremely fast in practice even for number fields of large degree (see section \ref{secKOM}). In section \ref{subsectwogen} we easily derive two-element representations of fractional ideals.

For any prime ideal $\p$ of $K$ we keep the notations for $\ty_\p$, $f_\p(x)$, $\phi_\p(x)$, $\ff{\p}$, $\t_\p$, $K_\p$, $\Z_{K_\p}$, as introduced in section \ref{secPadic}.

\subsection{Local generators of the prime ideals}\label{subsectPseudo}
\begin{definition}
A pseudo-generator of a prime ideal $\p$ of $K$ is an integral element $\pi\in\Z_K$ such that $v_\p(\pi)=1$.
\end{definition}

Let $p$ be a prime number, and let $\p=[p;\phi_1,\dots,\phi_r,\phi_\p]$ be a prime ideal factor of $p\Z_K$, corresponding to an $f$-complete type $\ty_\p$ with an added $(r+1)$-th level, as indicated in section \ref{subsecApprox}. In this section we show how to compute a pseudogenerator of $\p$ from the secondary invariants $u_i$,
$\op{Quot}_i$, of $\ty_\p$, for  $1\le i\le r$, computed along the flow of Montes algorithm as indicated in (\ref{secondaryinvariants}).

For each level $1\le i\le r$, let us denote:
$$\op{Quot}_i=[Q_{i,0}(x),\dots, Q_{i,e_i-1}(x)].
$$Recall that $Q_{i,0}(x)=1$, and for $0<j<e_i$, the polynomial $Q_{i,j}(x)\in\Z[x]$ is the  $(s_i-j)$-th quotient of the $\phi_i$-adic development of $f(x)$ (cf. (\ref{quotients})), where $s_i$ is the abscissa of the right end point of the side of slope $\lambda_i$ of $N^-_i(f)$. Also, let us define
$$H_{i,0}=0,\qquad H_{i,j}=\dfrac{u_i+j(|\lambda_i|+v_i(\phi_i))}{e_1\cdots e_{i-1}},
\quad \forall\,0<j<e_i.
$$

\begin{proposition}\label{quotientvalue}
For each level $1\le i\le r$ and subindex $\,0\le j<e_i$:
\begin{enumerate}
\item $v_\q(Q_{i,j}(\t))\ge e(\q/p)H_{i,j}$, for all prime ideals $\q\mid p$.
\item $v_\p(Q_{i,j}(\t))=e(\p/p)H_{i,j}$.
\end{enumerate}
\end{proposition}

\begin{proof}
Item 1 being proved in \cite[Prop.10]{GMNbasis}, let us prove item 2. Fix a level $1\le i\le r$ and a subindex $\,0\le j<e_i$. Let
$\ell_i=\ord_{\psi_{i-1}}R_{i-1}(f)$, and let $f(x)=\sum_{s\ge 0}a_s\phi_i^s$ be the $\phi_i$-adic development of $f(x)$.
The Newton polygon of $i$-th order of $f(x)$, $N_i(f)$, is the lower convex envelope of the cloud of points $(s,v_i(a_s\phi_i^s))$, for all $s\ge0$. The principal part $N_i^-(f)$ is equal to $N_i(f)\cap\left([0,\ell_i]\times\R\right)$; the typical shape of this polygon is illustrated in Figure 2. Let $S_{\lambda_i}$ be the side of slope $\lambda_i$ of this polygon, and $s_i$ the abscissa of the right end point of $S_{\lambda_i}$.

\begin{center}
\setlength{\unitlength}{5.mm}
\begin{picture}(20,11)
\put(-.15,8.85){$\bullet$}\put(1.85,5.85){$\bullet$}
\put(10.85,2.85){$\bullet$}\put(9.85,3.85){$\bullet$}
\put(8.85,3.5){$\circ$}\put(16.85,1.85){$\bullet$}
\put(0,-1.5){\line(0,1){12}}\put(-1,0){\line(1,0){19}}
\put(2,6.03){\line(-2,3){2}}\put(2,6){\line(-2,3){2}}
\put(2,6){\line(3,-1){9}}\put(2,6.03){\line(3,-1){9}}
\put(11,3){\line(6,-1){6}}\put(11,3.03){\line(6,-1){6}}
\multiput(2,6)(-.15,.05){15}{\mbox{\begin{scriptsize}.\end{scriptsize}}}
\multiput(11,3)(-.1,.1){10}{\mbox{\begin{scriptsize}.\end{scriptsize}}}
\put(8.5,6){\begin{footnotesize}$N_i^-(f)$\end{footnotesize}}
\put(5,5.2){\begin{footnotesize}$S_{\lambda_i}$\end{footnotesize}}
\multiput(17,-.6)(0,.25){11}{\vrule height2pt}
\multiput(11,-.1)(0,.25){13}{\vrule height2pt}
\multiput(10,-.6)(0,.25){19}{\vrule height2pt}
\multiput(9,-.1)(0,.25){15}{\vrule height2pt}
\put(10.15,.15){\begin{footnotesize}$t$\end{footnotesize}}
\put(11.2,.2){\begin{footnotesize}$s_i$\end{footnotesize}}
\put(17.2,.2){\begin{footnotesize}$\ell_i$\end{footnotesize}}
\put(8,-.6){\begin{footnotesize}$s_i-e_i$\end{footnotesize}}
\multiput(-.1,2)(.25,0){68}{\hbox to 2pt{\hrulefill }}
\put(-1.6,1.9){\begin{footnotesize}$v_i(f)$\end{footnotesize}}
\put(-.6,6.6){\begin{footnotesize}$H$\end{footnotesize}}
\put(.2,-.6){\begin{footnotesize}$0$\end{footnotesize}}
\put(13,-.6){\vector(-1,0){3}}\put(13,-.6){\vector(1,0){4}}
\put(12,-1.4){\begin{footnotesize}$N_i^-(q_t\phi_i^t)$\end{footnotesize}}
\end{picture}
\end{center}\vskip.4cm
\begin{center}
Figure 2
\end{center}

For any $0\le t\le \ell_i$, let $q_t(x)$ be the $t$-th quotient of the $\phi_i$-adic development (see (\ref{quotients})). We have  $f(x)=q_t(x)\phi_i(x)^t+r_t(x)$, with
$$
r_t(x)=\sum_{0\le s< t}a_s(x)\phi_i(x)^s,\qquad q_t(x)\phi_i(x)^t=\sum_{t\le s}a_s(x)\phi_i(x)^s.
$$
Hence, if $t_0$ is the smallest abscissa of a vertex of $N_i(f)$, such that $t_0\ge t$, we have
$$
N_i(q_t\phi_i^t)\cap \left([t_0,\infty)\times \R\right)=
N_i(f)\cap \left([t_0,\infty)\times \R\right).$$
Recall that $Q_{i,j}(x)=q_t(x)$, for $t=s_i-j$; for this value of $t$ we have $t_0=s_i$ (see Figure 2). On the other hand, all points in the cloud $(s,v_i(a_s\phi_i^s))$, for $s_i-e_i<s<s_i$ lie strictly above $S_{\lambda_i}$, because the point on $S_{\lambda_i}$ with integer coordinates and closest to the right end point has abscissa $s_i-e_i$. Hence, the line of slope $\lambda_i$ that first touches $N_i(q_{s_i-j}\phi_i^{s_i-j})$ from below is the line containg $S_{\lambda_i}$.
This line has ordinate at the origin (see Figure 2):
$$H=v_i\left(a_{s_i}\phi_i^{s_i}\right)+s_i|\lambda_i|=u_i+s_i(v_i(\phi_i)+|\lambda_i|).$$
On the other hand, this line touches $N_i(q_{s_i-j}\phi_i^{s_i-j})$ only at the point
$(s_i,v_i(a_{s_i}\phi_i^{s_i}))$, so that $R_i(q_{s_i-j}\phi_i^{s_i-j})(y)$ is a constant and $\op{Trunc}_i(\ty_\p)$ does not divide $q_{s_i-j}\phi_i^{s_i-j}$. Therefore, Proposition \ref{vgt} shows that
$$
v(q_{s_i-j}(\t_\p)\phi_i(\t_\p)^{s_i-j})=\dfrac {u_i+s_i(v_i(\phi_i)+|\lambda_i|)}{e_1\cdots e_{i-1}}.
$$
By the Theorem of the polygon \cite[Thm.3.1]{HN},
\begin{equation}\label{thmpolygon}
v(\phi_i(\t_\p))=\dfrac{v_i(\phi_i)+|\lambda_i|}{e_1\cdots e_{i-1}},
\end{equation} so that
$$
v(q_{s_i-j}(\t_\p))=\dfrac {u_i+j(v_i(\phi_i)+|\lambda_i|)}{e_1\cdots e_{i-1}}=H_{i,j}.
$$
By (\ref{tautology}), $v_\p(q_{s_i-j}(\t))=e(\p/p)H_{i,j}$.
\end{proof}

If $e(\p/p)=1$, then $\pi_\p:=p$ is a pseudo-generator of $\p$. If $e(\p/p)>1$, we can always find a pseudo-generator of $\p$ by computing a suitable product of quotients in the lists $\op{Quot}_i$, divided by a suitable power of $p$.

\begin{corollary}\label{products}	\mbox{\null}
\begin{enumerate}
\item Let $j_1,\dots,j_r$ be subindices satisfying, $0\le j_i<e_i$, for all $1\le i\le r$. Then, the following element belongs to $\Z_K$:
$$\pi_{j_1,\dots,j_r}:=Q_{1,j_1}(\t)\cdots Q_{r,j_r}(\t)/p^{\lfloor H_{1,j_1}+\cdots+H_{r,j_r}\rfloor}.$$
\item If $e(\p/p)>1$, there is a unique family  $j_1,\dots,j_r$ as above, for which
$v_\p(\pi_{j_1,\dots,j_r})=1$. This family may be recursively computed as follows:
$$\as{1.2}
\begin{array}{l}
 j_r\equiv h_r^{-1} \md{e_r},\\ \op{res}_r:=(j_rh_r-1)/e_r,\\
j_{r-1}\equiv  -h_{r-1}^{-1}(u_r+j_rv_r(\phi_r)+\op{res}_r) \md{e_{r-1}},\\ \op{res}_{r-1}:=(j_{r-1}h_{r-1}+u_r+j_rv_r(\phi_r)+\op{res}_r)/e_{r-1},\\
\qquad \cdots\qquad \cdots\\
j_1\equiv  -h_1^{-1}(u_2+j_2v_2(\phi_2)+\op{res}_2) \md{e_1}.
\end{array}
$$
\end{enumerate}
\end{corollary}

\begin{proof}
Item 1 is an immediate consequence of item 1 of Proposition
\ref{quotientvalue}.

Also, by Proposition \ref{quotientvalue},
$$
v_\p(\pi_{j_1,\dots,j_r})=e(\p/p)\left(H_{1,j_1}+\cdots+H_{r,j_r}-\lfloor H_{1,j_1}+\cdots+H_{r,j_r}\rfloor\right).
$$
Thus, item 2 states that there is a unique family $j_1,\dots,j_r$ such that
$$
H_{1,j_1}+\cdots+H_{r,j_r}\equiv \dfrac 1{e(\p/p)} \md{\Z}.
$$
Since $e(\p/p)=e_1\cdots e_r$ and $|\lambda_i|=h_i/e_i$, this is equivalent to:
\begin{align*}
u_1+j_1v_1(\phi_1)+&\dfrac{j_1h_1+u_2+j_2v_2(\phi_2)}{e_1}+\cdots\\
&\cdots+\dfrac{j_{r-1}h_{r-1}+u_r+j_rv_r(\phi_r)}{e_1\cdots e_{r-1}}+\dfrac{j_rh_r}{e_1\cdots e_r}\equiv \dfrac{1}{e_1\cdots e_r}  \md{\Z}.
\end{align*}
Clearly this congruence has a unique solution $j_1, \dots, j_r$ satisfying $0\le j_i<e_i$, for all $1\le i\le r$, and this solution may be recursively obtained by the procedure described in item 2.
\end{proof}

The only property of $\ty_\p$ that we used in in Corollary \ref{products} is: $e_1\cdots e_r=e(\p/p)$. Thus, we don't need to use all levels of $\ty_\p$ to compute a pseudo-generator of $\p$; in practice we take $r$ to be the minimum level such that
$e_1\cdots e_r=e(\p/p)$.

\subsection{Generators of the prime ideals}\label{subsectgenerators}
Let $p$ be a prime number, and $\P$ the set of prime ideals of $K$ lying over $p$. Once we have pseudo-generators $\pi_\p\in\p$ of all $\p\in\P$, in order to find generators we need only to compute a family of integral elements, $\{b_\p\in \Z_K\}_{\p\in\P}$, satisfying:
\begin{equation}\label{bp}
v_\p(b_\p)=0,\ \forall\,\p\in\P,\qquad v_\q(b_\p)>1,\ \forall\,\q,\p\in\P,\ \q\ne\p.
\end{equation}
Then, for each $\p\in\P$, the integral element:
$$
\alpha_\p:=b_\p\pi_\p+\sum_{\q\in\P, \q\ne\p} b_\q\in\Z_K
$$
clearly satisfies: $v_\p(\alpha_\p)=1$, $v_\q(\alpha_\p)=0$, for all $\q\ne \p$. Therefore, $\p$ is the ideal generated by $p$ and $\alpha_\p$.
The rest of this section is devoted to the construction of these multipliers $\{b_\p\}$.

Let $(\ty_\p)_{\p\in\P}$ be the parameterization of the set $\P$ by a family of $f$-complete types obtained by an application of Montes algorithm. As usual, we suppose that each $\ty_\p$ has been conveniently enlarged with an $(r_\p+1)$-th level, as indicated in section \ref{subsecApprox}. From now on we provide the invariants of $\ty_\p$ with a subscript $\p$ to distinguish the prime ideal they belong to: $r_\p, \phi_{i,\p}, m_{i,\p}, \la_{i,\p},$ etc.

The integral elements $b_\p$ will be constructed as suitable products of $\phi$-polynomials divided by suitable powers of $p$. The crucial ingredient is Proposition \ref{vpq}, that computes $v_\p(\phi_{i,\q}(\t))$ for all $\p\ne\q$ in $\P$, and all $1\le i\le r_\q+1$.

\begin{definition}
For any pair $\p,\q\in\P$, we define the \emph{index of coincidence} between the types  $\ty_\p$
and $\ty_\q$ as:
$$
i(\ty_\p,\ty_\q)=\left\{\begin{array}{ll}
0,&\mbox{if }\psi_{0,\p}\ne\psi_{0,\q},\\
\min\left\{j\in\Z_{>0}\tq (\phi_{j,\p},\lambda_{j,\p},\psi_{j,\p})\ne
(\phi_{j,\q},\lambda_{j,\q},\psi_{j,\q})\right\},&\mbox{if }\psi_{0,\p}=\psi_{0,\q}.
\end{array}
\right.
$$
Alternatively, $i(\ty_\p,\ty_\q)$ is the least subindex $j$ for which $\op{Trunc}_j(\ty_\p)\ne \op{Trunc}_j(\ty_\q)$.
\end{definition}

\begin{remark}\label{coincide}\mbox{\null}
By definition,
$$\phi_{i,\p}=\phi_{i,\q},\quad \lambda_{i,\p}=\lambda_{i,\q},\quad \psi_{i,\p}=\psi_{i,\q}, \quad\forall\,i< i(\ty_\p,\ty_\q).$$  Hence, by the definition of the $p$-adic valuations $v_{i,\p}$, $v_{i,\q}$, and by \cite[Thm. 2.11]{HN}, we get:
$$v_{i,\p}=v_{i,\q}, \quad m_{i,\p}=m_{i,\q}, \quad
v_{i,\p}(\phi_{i,\p})=v_{i,\q}(\phi_{i,\q}),\quad \forall\,i\le i(\ty_\p,\ty_\q).
$$
\end{remark}

\begin{lemma}\label{lessthanr}
If  $\p,\q\in\P$, and $\p\ne\q$, then $i(\ty_\p,\ty_\q)\le\min\{r_\p,r_\q\}$.
\end{lemma}

\begin{proof}
Suppose $r_\p\le r_\q$ and $i(\ty_\p,\ty_\q)=r_\p+1$. Then, $\op{Trunc}_{r_\p}(\ty_\p)=\op{Trunc}_{r_\p}(\ty_\q)$. Since the type $\op{Trunc}_{r_\p}(\ty_\p)$ is $f$-complete, it singles out a unique $p$-adic irreducible factor of $f(x)$ (item 2 of Definition \ref{defs}). Hence, $\op{Trunc}_{r_\p}(\ty_\q)$ is also $f$-complete and it singles out the same $p$-adic irreducible factor of $f(x)$. This implies that $\p=\q$.
\end{proof}

By (\ref{tautology}) and (\ref{thmpolygon}), for all $1\le i\le r_\p+1$ we have
\begin{equation}\label{p=q}
v_\p(\phi_{i,\p}(\t))=e(\p/p)\,\dfrac{v_{i,\p}(\phi_{i,\p})+|\lambda_{i,\p}|}{e_{1,\p}\cdots e_{i-1,\p}}.
\end{equation}

In order to compute the values of $v_\p(\phi_{i,\q}(\t))$, for $\p\ne\q$, we need still another definition.
	
\begin{definition}\label{gcphi}
Let $\p,\q\in\P$, $\q\ne\p$, and $j=i(\ty_\p,\ty_\q)$. Let $s_\p=\#\op{Refinements}_{j,\p}$, and consider the list $\op{Ref}_\p$ obtained by extending the list $\op{Refinements}_{j,\p}$ by adding the pair $[\phi_{j,\p},\lambda_{j,\p}]$ at the last position:
$$
\op{Ref}_\p=\left[\left[\phi_{j,\p}^{(1)},\lambda_{j,\p}^{(1)}\right], \cdots ,\left[\phi_{j,\p}^{(s_\p)},\lambda_{j,\p}^{(s_\p)}\right],\left[\phi_{j,\p}^{(s_\p+1)},\lambda_{j,\p}^{(s_\p+1)}\right]:=\left[\phi_{j,\p},\lambda_{j,\p}\right]\right].
$$
Let $\op{Ref}_\q$ be the analogous list for the prime ideal $\q$.

We define the \emph{greatest common $\phi$-polynomial} of the pair $(\ty_\p,\ty_q)$ to be the more advanced common $\phi$-polynomial in the two lists $\op{Ref}_\p$, $\op{Ref}_\q$. We denote it by:
$$
\phi(\p,\q):=\phi_{j,\p}^{(k)}=\phi_{j,\q}^{(k)},
$$
for the maximum index $k$ such that  $\phi_{j,\p}^{(k)}=\phi_{j,\q}^{(k)}$.

We define the \emph{hidden slopes} of the pair $(\ty_\p,\ty_\q)$ to be: $\lambda_\p^\q:=\lambda_{j,\p}^{(k)}$, $\lambda_\q^\p:=\lambda_{j,\q}^{(k)}$.
\end{definition}

\noindent{\bf Remarks. }\medskip

(1) \ By the concrete way the processes of branching, enlarging and/or refining were defined, this polynomial $\phi(\p,\q)$ always exists. In fact, let us show that we must have $\phi_{j,\p}^{(1)}=\phi_{j,\q}^{(1)}$.
Since $\op{Trunc}_{j-1}(\ty_\p)=\op{Trunc}_{j-1}(\ty_\q)$, this type had some original representative (say) $\phi_j$. By considering $N_{\phi_j,v_j}^-(f)$ and the irreducible factors of all residual polynomials of all sides, we had different branches $(\lambda,\psi)$ to analyze; if there was only one branch, the algorithm necessarily performed a refinement step, because otherwise we would have $i(\ty_\p,\ty_\q)>j$. After eventually a finite number of these unibranch refinement steps (that were not stored in the list $\op{Refinements}_j$), we considered some representative, let us call it $\phi_j$ again, leading to several branches. One of these branches led later to the type $\ty_\p$ and one of them (maybe still the same) to the type $\ty_\q$. If the $\p$-branch experimented refinement, the list $\op{Refinements}_{j,\p}$ had $\phi_{j,\p}^{(1)}=\phi_j$ as its initial $\phi$-polynomial; if the $\p$-branch was $f$-complete or had to be enlarged, then the list $\op{Refinements}_{j,\p}$ remained empty and we had $\phi_{j,\p}=\phi_j$. In any case, $\phi_j$ is the first $\phi$-polynomial of the list $\op{Ref}_\p$.\medskip

(2) \ All $\phi_{j,\p}^{(\ell)}$, $\phi_{j,\q}^{(\ell)}$ are representatives of  $\ty_{j-1}:=\op{Trunc}_{j-1}(\ty_\p)=\op{Trunc}_{j-1}(\ty_\q)$; in particular, all these polynomials have degree $m_j$. With the obvious meaning for $\psi_{j,\p}^{( \ell)}$, we have necessarily:
$$
\left[\phi_{j,\p}^{(k)},\lambda_{j,\p}^{(k)},\psi_{j,\p}^{(k)}\right]\ne
\left[\phi_{j,\q}^{(k)},\lambda_{j,\q}^{(k)},\psi_{j,\q}^{(k)}\right],\quad
\left[\phi_{j,\p}^{(\ell)},\lambda_{j,\p}^{(\ell)},\psi_{j,\p}^{(\ell)}\right]=
\left[\phi_{j,\q}^{(\ell)},\lambda_{j,\q}^{(\ell)},\psi_{j,\q}^{(\ell)}\right],
$$
for all $1\le \ell<k$. Thus, $\phi(\p,\q)$ is the first representative of $\ty_{j-1}$ for which the branches of $\ty_\p$ and  $\ty_\q$ are different.\medskip

(3) \ Caution: we may have $\op{Ref}_\p=\op{Ref}_\q$. In this case $\phi(\p,\q)=\phi_{j,\p}= \phi_{j,\q}$ and $\lambda_\p^\q=\lambda_{j,\p}=\lambda_{j,\q}=\lambda_\q^\p$; the branches of $\ty_\p$ and $\ty_\q$ are distinguished by $\psi_{j,\p}\ne\psi_{j,\q}$.    \medskip

\begin{proposition}\label{vpq}
Let $\p,\q\in\P$, $\p\ne \q$, and $j=i(\ty_\p,\ty_\q)$. Let $\phi(\p,\q)$ be the greatest common $\phi$-polynomial  of the pair $(\ty_\p,\ty_\q)$ and $\lambda_\p^\q$, $\lambda_\q^\p$ the hidden slopes. For any $1\le i\le r_\q+1$,
$$
\as{2.2}
\dfrac{v_\p(\phi_{i,\q}(\t))}{e(\p/p)}=\left\{
\begin{array}{ll}
0,&\mbox{if }j=0,\\
\dfrac{v_i(\phi_i)+|\lambda_i|}{e_1\cdots e_{i-1}}
,&\mbox{if }i<j,\\
\dfrac{v_j(\phi_j)+|\lambda_\p^\q|}{e_1\cdots e_{j-1}},&\mbox{if $i=j$ and }\, \phi_{j,\q}=\phi(\p,\q),\\
\dfrac{v_j(\phi_j)+\min\{|\lambda_\p^\q|,|\lambda_\q^\p|\}}{e_1\cdots e_{j-1}},&\mbox{if $i=j$ and }\phi_{j,\q}\ne \phi(\p,\q),\\
\dfrac{m_{i,\q}}{m_j}\cdot\dfrac{v_j(\phi_j)+\min\{|\lambda_\p^\q|,|\lambda_\q^\p|\}}{e_1\cdots e_{j-1}},&\mbox{if }i>j>0.
\end {array}
\right.
$$

In these formulas we omit the subscripts $\p$, $\q$ when the invariants of the two types coincide (cf. Remark \ref{coincide}).

\end{proposition}

\begin{proof}
The case $j=0$ was seen in (\ref{v=0}). The cases $i<j$ and $i=j$, $\phi_{j,\q}=\phi(\p,\q)$, are a consequence of (\ref{p=q}).

Suppose $i>j>0$ and $\phi_{j,\p}=\phi_{j,\q}$; we have then, $\phi(\p,\q)=\phi_{j,\p}=\phi_{j,\q}$ and
$\lambda_\p^\q=\lambda_{j,\p}$, $\lambda_\q^\p=\lambda_{j,\q}$. We compute $v_\p(\phi_{i,\q}(\t))$ by applying Proposition \ref{vgt} to the polynomial
$g(x)=\phi_{i,\q}(x)$ and the type $\ty_\p$. Since $v_{j,\p}=v_{j,\q}$ and $\phi_{j,\p}=\phi_{j,\q}$, we have $N_{j,\p}(\phi_{i,\q})=N_{j,\q}(\phi_{i,\q})$.
On the other hand, we saw in the proof of Lemma \ref{vjphii} that $N_{j,\q}(\phi_{i,\q})$ is one-sided of slope $\lambda_{j,\q}$. Figure 3 shows the three possibilities for the line $L_{\lambda_{j,\p}}$ of slope $\lambda_{j,\p}$ that first touches  $N_{j,\p}(\phi_{i,\q})$ from below.

\begin{center}
\setlength{\unitlength}{5.mm}
\begin{picture}(22,5.5)
\put(.8,3.8){$\bullet$}\put(3.8,.8){$\bullet$}
\put(0,0){\line(1,0){6}}\put(1,-1){\line(0,1){6}}
\put(4,1){\line(-1,1){3}}\put(4.02,1){\line(-1,1){3}}
\multiput(.45,4.85)(.05,-.1){30}{\mbox{\begin{scriptsize}.\end{scriptsize}}}
\multiput(4,-.1)(0,.25){5}{\vrule height1pt}
\multiput(.9,1)(.25,0){12}{\hbox to 2pt{\hrulefill }}
\put(.55,-.6){\begin{footnotesize}$0$\end{footnotesize}}
\put(-1.3,.9){\begin{footnotesize}$v_j(\phi_{i,\q})$\end{footnotesize}}
\put(3,-.6){\begin{footnotesize}$m_{i,\q}/m_j$\end{footnotesize}}
\put(2.6,2.6){\begin{footnotesize}$\la_{j,\q}$\end{footnotesize}}
\put(1.65,1.6){\begin{footnotesize}$L_{\lambda_{j,\p}}$\end{footnotesize}}
\put(.1,3.8){\begin{footnotesize}$H$\end{footnotesize}}
\put(1.2,-1.6){$\la_{j,\p}<\la_{j,\q}$}
\put(11.8,.8){$\bullet$}\put(8.85,3.8){$\bullet$}
\put(8,0){\line(1,0){6}}\put(9,-1){\line(0,1){6}}
\put(12,1){\line(-1,1){3}}\put(12.02,1){\line(-1,1){3}}
\multiput(9,3.9)(-.1,.1){7}{\mbox{\begin{scriptsize}.\end{scriptsize}}}
\multiput(12,.9)(.1,-.1){7}{\mbox{\begin{scriptsize}.\end{scriptsize}}}
\multiput(12,-.1)(0,.25){5}{\vrule height1pt}
\multiput(8.9,1)(.25,0){12}{\hbox to 2pt{\hrulefill }}
\put(8.55,-.6){\begin{footnotesize}$0$\end{footnotesize}}
\put(6.7,.9){\begin{footnotesize}$v_j(\phi_{i,\q})$\end{footnotesize}}
\put(11,-.6){\begin{footnotesize}$m_{i,\q}/m_j$\end{footnotesize}}
\put(10.6,2.6){\begin{footnotesize}$\la_{j,\q}$\end{footnotesize}}
\put(12.8,.5){\begin{footnotesize}$L_{\lambda_{j,\p}}$\end{footnotesize}}
\put(8.2,3.8){\begin{footnotesize}$H$\end{footnotesize}}
\put(9.2,-1.6){$\la_{j,\p}=\la_{j,\q}$}
\put(19.8,.8){$\bullet$}\put(16.85,3.8){$\bullet$}
\put(16,0){\line(1,0){6}}\put(17,-1){\line(0,1){6}}
\put(20,1){\line(-1,1){3}}\put(20.02,1){\line(-1,1){3}}
\multiput(21,.4)(-.1,.05){47}{\mbox{\begin{scriptsize}.\end{scriptsize}}}
\multiput(20,-.1)(0,.25){5}{\vrule height1pt}
\multiput(16.9,1)(.25,0){12}{\hbox to 2pt{\hrulefill }}
\put(16.55,-.6){\begin{footnotesize}$0$\end{footnotesize}}
\put(14.7,.9){\begin{footnotesize}$v_j(\phi_{i,\q})$\end{footnotesize}}
\put(19,-.6){\begin{footnotesize}$m_{i,\q}/m_j$\end{footnotesize}}
\put(18.6,2.6){\begin{footnotesize}$\la_{j,\q}$\end{footnotesize}}
\put(20.6,.9){\begin{footnotesize}$L_{\lambda_{j,\p}}$\end{footnotesize}}
\put(16.2,2.2){\begin{footnotesize}$H$\end{footnotesize}}
\put(16.9,2.5){\line(1,0){.2}}
\put(17.2,-1.6){$\la_{j,\p}>\la_{j,\q}$}
\end{picture}
\end{center}\bigskip
\begin{center}
Figure 3
\end{center}

A glance at Figure 3 shows that
$$
H=v_j(\phi_{i,\q})+\dfrac{m_{i,\q}}{m_j}\,\min\{|\lambda_{j,\p}|,|\lambda_{j,\q}|\}=
\dfrac{m_{i,\q}}{m_j}\,(v_j(\phi_j)+\min\{|\lambda_{j,\p}|,|\lambda_{j,\q}|\}),
$$
the last equality by Lemma \ref{vjphii}. Now, if $\lambda_{j,\p}\ne \lambda_{j,\q}$, the line $L_{\lambda_{j,\p}}$ touches the polygon only at one point, and the residual polynomial $R_{j,\p}(\phi_{i,\q})(y)$ is a cons\-tant. If  $\lambda_{j,\p}=\lambda_{j,\q}$, then
 $L_{\lambda_{j,\p}}$ contains $N_{j,\p}(\phi_{i,\q})$ and $R_{j,\p}(\phi_{i,\q})=
R_{j,\q}(\phi_{i,\q})$ is a power of $\psi_{j,\q}$, up to a multiplicative constant. In this case, necessarily $\psi_{j,\q}\ne\psi_{j,\p}$, by the definition of $j=i(\ty_\p,\ty_\q)$. Therefore, $\op{Trunc}_j(\ty_\p)$ never divides $\phi_{i,\q}$, and Proposition \ref{vgt} shows that $v(\phi_{i,\q}(\t_\p))=H/(e_1\cdots e_{j-1})$. By (\ref{tautology}), we get the desired expression for $v_\p(\phi_{i,\q}(\t))$.

Suppose now $i>j>0$, $\phi_{j,\p}\ne\phi_{j,\q}$, or $i=j$, $\phi_{j,\q}\ne\phi(\p,\q)$.
Consider a new type $\tilde{\ty}_\p$, constructed as follows: if $\phi_{j,\p}=\phi(\p,\q)$, we take $\tilde{\ty}_\p=\ty_\p$, and if $\phi_{j,\p}\ne\phi(\p,\q)$, we take
\begin{multline*}
\tilde{\ty}_\p=(\phi_1;\lambda_1,\phi_2;\cdots,\phi_{j-1};\lambda_{j-1},\phi(\p,\q);\lambda^\q_\p,\phi_{j,\p};\\\lambda_{j,\p}-\lambda^\q_\p,\phi_{j+1,\p};\lambda_{j+1,\p},\cdots,\phi_{r_\p+1,\p};\lambda_{r_\p+1,\p},\psi_{r_\p+1,\p}).
\end{multline*}
By \cite[Cor.3.6]{GMNalgorithm}, $\tilde{\ty}_\p$ is a type, and it is also $f$-complete. If $\phi_{j,\p}\ne\phi(\p,\q)$ then $\tilde{\ty}_\p$ is not optimal because $\deg\phi(\p,\q)=\deg \phi_{j,\p}$, but optimality is not necessary to apply Proposition \ref{vgt}.

We consider an analogous construction for $\tilde{\ty}_\q$. The new types satisfy $i(\tilde{\ty}_\p,\tilde{\ty}_\q)=j$ and they have the same $j$-th $\phi$-polynomial; finally if $i=j$, the polynomial $\phi_{i,\q}$ is the $(j+1)$-th $\phi$-polynomial of $\tilde{\ty}_\q$.
Therefore, the computation of $v_\p(\phi_{i,\q})$ is deduced by the same arguments as above.
\end{proof}

We are ready to construct the family $\{b_\p\}_{\p\in\P}$. Consider the following equivalence relation in the set $\P$:
$$
\p\sim\q\sii \psi_{0,\p}=\psi_{0,\q},
$$and denote by $[\p]$ the class of any $\p\in\P$.

For each class $[\p]$, let $\phi_{1,[\p]}(x)\in\Z[x]$ be the first $\phi$-polynomial in any list $\op{Ref}_\q$  for some $\q\in[\p]$; we saw in the first remark following Definition \ref{gcphi} that all lists $\op{Ref}_\q$,  for $\q\in[\p]$, have the same initial $\phi$-polynomial. Now, for each $\q\in[\p]$, denote by $\lambda_{1,\q}^0$
the first slope in the list $\op{Ref}_\q$. In other words, for any $\q\in[\p]$, we have
$$\as{1.6}
(\phi_{1,[\p]},\lambda_{1,\q}^0)=
\left\{
\begin{array}{ll}
(\phi_{1,\q},\lambda_{1,\q}),&\mbox { if }\op{Refinements}_{1,\q}\mbox{ is empty},\\
\left(\phi_{1,\q}^{(1)},\lambda_{1,\q}^{(1)}\right),&\mbox { if }\op{Refinements}_{1,\q}\mbox{ is not empty}.
\end{array}
\right.
$$
By (\ref{v=0}) and  (\ref{thmpolygon}),
\begin{equation}\label{vqphip}\as{1.2}
v_\q(\phi_{1,[\p]}(\t))=\left\{
\begin{array}{ll}
0,&\mbox { if }\q\not\in[\p],\\
e(\q/p)|\lambda^0_{1,\q}|,&\mbox { if }\q\in[\p].
\end{array}
\right.
\end{equation}
Consider now, for each class $[\p]$:
$$
B_{[\p]}(x):=\prod_{[\q]\ne[\p]}\phi_{1,[\q]}(x).
$$

Fix a prime ideal $\p\in\P$. If $\#[\p]=1$, the element $b_\p=(B_{[\p]}(\t))^2$ satisfies (\ref{bp})
already. Suppose now $\#[\p]>1$. For all $\l\in[\p]$, $\l\ne\p$, let $\phi_{\l}=\phi_{r_{\l}+1,\l}$ be the Montes approximation to $f_{\l}(x)$ contained in $\ty_{\l}$; consider the least positive numerator and denominator of the rational number $v_\p(\phi_{\l}(\t))/e(\p/p)$ (which has been computed in Proposition \ref{vpq}):
$$\dfrac{v_\p(\phi_{\l}(\t))}{e(\p/p)}=\dfrac{n_{\l}}{d_{\l}},\quad \op{gcd}(n_{\l},d_{\l})=1.
$$
We look for an integral element of the form:
\begin{equation}\label{finalbp}
b_\p=\dfrac{(B_{[\p]}(\t))^m\prod_{\l\in[\p],\,\l\ne\p}\phi_{\l}(\t)^{d_{\l}}}{p^N},
\end{equation}
where the exponents $N,m$ are given by
$$N=\sum_{\l\in[\p],\,\l\ne\p}n_{\l},\qquad m=\left\lceil\max_{\q\in\P,\,\q\not\in[\p]}\left\{\dfrac{Ne(\q/p)+2}{e(\q/p)|\lambda^0_{1,\q}|}\right\}\right\rceil.$$

Take $\q\not\in[\p]$. By (\ref{v=0}) and (\ref{vqphip}), we have $$v_\q\left(\prod_{\l\in[\p],\,\l\ne\p}\phi_{\l}(\t)^{d_{\l}}\right)=0,\qquad
v_\q(B_{[\p]}(\t))=v_\q(\phi_{1,[\q]}(\t))=e(\q/p)|\lambda^0_{1,\q}|.$$
Hence, $v_\q(p^Nb_\p)=m\,e(\q/p)|\lambda^0_{1,\q}|\ge  Ne(\q/p)+2$,
so that $v_\q(b_\p)>1$, as desired.

For the prime $\p$ itself, we have $v_\p(B_{[\p]}(\t))=0$ and, by construction,
$$
v_\p(b_\p)=\left(\sum_{\l\in[\p],\,\l\ne\p}d_{\l}v_\p(\phi_{\l}(\t))\right)-Ne(\p/p)=0.
$$

Finally, for a prime $\l\in[\p]$, $\l\ne\p$, we have  $v_{\l}(B_{[\p]}(\t))=0$ and
$$
v_{\l}(b_\p)=V_1+V_2-Ne(\l/p),\quad V_1:=\sum_{\l'\in[\p],\,\l'\ne\p,\l}v_{\l}(\phi_{\l'}(\t)^{d_{\l'}}),\quad V_2:=v_{\l}(\phi_{\l}(\t)^{d_{\l}}).
$$
Lemma \ref{lessthanr} and Proposition \ref{vpq} show that $V_1$ (as all invariants we used so far) depends only on the numerical invariants of the types $\ty_{\l}, \ty_{\l'}$, of level $1\le i\le i(\ty_{\l}, \ty_{\l'})\le \min\{r_{\l},r_{\l'}\}$, and not on the quality of the Montes approximations $\phi_{\l'}$. On the other hand, $V_2$ depends on the choice of $\phi_{\l}$ as a Montes approximation of $f_{\l}(x)$; by (\ref{p=q}):
$$
V_2=v_{\l}(\phi_{\l}(\t)^{d_{\l}})=d_{\l}\,e(\l/p)\dfrac{v_{r_{\l}+1,\l}(\phi_{\l})+h_{r_{\l}+1,\l}}{e_{1,\l}\cdots e_{r_{\l},\l}}=d_{\l}(v_{r_{\l}+1,\l}(\phi_{\l})+h_{r_{\l}+1,\l}).
$$
Hence, for all $\l\in[\p]$, $\l\ne\p$, we improve the Montes approximation $\phi_{\l}$ till we get
$$
h_{r_{\l}+1,\l}\ge \dfrac{2+Ne(\l/p)-V_1}{d_\l}-v_{r_{\l}+1,\l}(\phi_{\l}).
$$This ensures that $v_{\l}(b_\p)>1$, as desired.

\subsection{Two-element representation of a fractional ideal}\label{subsectwogen}
Any fractional ideal $\a$ of $K$ admits a two-element representation: $\a=(\ell,\alpha)$, where $\alpha\in K$ and $\ell=\ell(\a)\in\Q$ is the least positive rational number contained in $\a$. It is straightforward to obtain such a representation from the two-element representation of the prime ideals obtained in the last section.
For the sake of completeness we briefly describe the routine. For each prime ideal $\p$ of $K$, we have computed an integral element $\alpha_\p\in\Z_K$ such that
$$
v_\p(\alpha_\p)=1,\quad v_\q(\alpha_\p)=0, \ \forall\,\q|p, \,\q\ne\p.
$$
These elements are of the form: $\alpha_\p=p^{\nu_\p}h_\p(\t)$, where $\nu_\p\in\Z$ and $h_\p(x)\in\Z[x]$ is a primitive polynomial. Generically, $\nu_\p\le0$, except for the special case $\p=p\Z_K$, where $\nu_\p=1$, $h_\p(x)=1$. Let us write
$$
\op{N}_{K/\Q}(h_\p(\t))=p^{\mu_\p}N_\p, \mbox{ with }p\nmid N_\p.
$$

Suppose first that $\a=\prod_{\p|p}\p^{a_\p}$ has support only in prime ideals dividing $p$. We take then:
$$
\alpha=\prod_{\p|p}\alpha_\p^{a_\p}\prod_{a_\p<0}N_\p^{|a_\p|},\qquad H:=\left\lceil\max_{\p|p}\left\{ \dfrac{a_\p}{e(\p/p)}\right\}\right\rceil.
$$
One checks easily that $\ell(\a)=p^H$ and $\a=(p^H,\alpha)$.

In the general case, we write $\a=\prod_{p\in\pp}\a_p$, where $\pp$ is a finite set of prime numbers and $\a_p$ is divided only by prime ideals lying over $p$. For each $\a_p$ we find a two-element representation $\a_p=(p^{H_p},\alpha_p)$; then, the two-element representation of $\a$ is:
$$
\a=\left(\prod_{p\in\pp}p^{H_p},\,\sum_{p\in\pp}\left(\prod_{q\in\pp,\,q\ne p}
q^{H_q+1}\right)\alpha_p\right).
$$
Note that the second generator $\alpha$ constructed in this way satisfies: $v_\p(\alpha)=v_\p(\a)$, for all $\p$ with $v_\p(\a)\ne 0$, which is slightly stronger than the condition $\a=(\ell,\alpha)$.

\section{Residue classes and Chinese remainder theorem}\label{secCRT}
In this section we show how to compute residue classes modulo prime ideals, and we design a chinese remainder theorem routine. As in the previous sections, this will be done without constructing (a basis of) the maximal order of $K$ and without the necessity to invert elements in the number field. Only some inversions in the finite residue fields are required.

For any prime ideal $\p$ of $K$ we keep the notations for $\ty_\p$, $f_\p(x)$, $\phi_\p(x)$, $\ff{\p}$, $\t_\p$, $K_\p$, $\Z_{K_\p}$, as introduced in section \ref{secPadic}.

\subsection{Residue classes modulo a prime ideal}
\label{redmap}
Let $\p$ be a prime ideal of $K$, corresponding to an $f$-complete type $\ty_\p$ with an added $(r+1)$-th level, as indicated in section \ref{subsecApprox}.

The finite field $\ff{\p}:=\ff{r+1}$ may be considered as a computational representation of the residue field $\Z_K/\p$. In fact, fix the topological embedding, $\iota_\p\colon K\hookrightarrow K_\p$, determined by sending $\t$ to $\t_\p$, and  consider the reduction modulo $\p$ map obtained by composition of the embedding
$\Z_K\hookrightarrow \Z_{K_\p}$ with the local reduction map constructed in (\ref{lred}):
$$
\op{red}_\p\colon \Z_K\hookrightarrow \Z_{K_\p}\stackrel{\op{lred}_\p}\lra \ff{\p}.
$$
The commutative diagram:
$$
\begin{array}{ccccc}
\Z_K&\hookrightarrow&\Z_{K_\p}&\stackrel{\op{lred}_\p}\lra&\ff{\p}\\
\downarrow&&\downarrow&&\parallel\\
\Z_K/\p&\iso &\Z_{K_\p}/\p\Z_{K_\p}&\stackrel{\gamma^{-1}}\iso&\ff{\p}
\end{array}
$$
shows that our reduction map $\op{red}_\p$ coincides with the canonical reduction map, $\Z_k\lra \Z_K/\p$, up to certain isomorphism $\Z_K/\p\iso \ff{\p}$.

The problem has now a computational perspective; we want to find a routine that computes $\op{red}_\p(\alpha)\in \ff{\p}$ for any given integral element $\alpha\in\Z_K$. To this end, it is sufficient to have a routine that computes $\op{lred}_\p(\alpha)\in\ff{\p}$, for any $\p$-integral $\alpha\in K$. Let us show that this latter routine may be based on item 2 of Proposition \ref{vgt} and Lemma \ref{prodgammas}.

Any $\alpha\in\Z_K$ can be written in a unique way as:
$$
\alpha=\dfrac ab\,\dfrac{g(\t)}{p^N},
$$
where $a,b$ are positive coprime integers not divisible by $p$ and $g(x)\in\Z[x]$ is a primitive polynomial. Clearly, $$\op{red}_\p(\alpha)=\op{lred}_\p(\iota_\p(\alpha))=\op{lred}_\p(a/b)\op{lred}_\p(g(\t_\p)/p^N),$$ and $\op{lred}_\p(a/b)\in\ff{\p}$ is the element in the prime field determined by the quotient of the classes modulo $p$ of $a$ and $b$. Thus, we need only to compute $\op{lred}_\p(g(\t_\p)/p^N)$.

If $N=0$, then $\op{lred}_\p(g(\t_\p))=\op{red}_\p(g(\t))$ is just the class of $g(x)$ modulo the ideal $(p,\phi_{1,\p}(x))$. In other words, if $\overline{g}(x)\in\ff{0}[x]$ is the polynomial obtained by reduction of the coefficients of $g(x)$ modulo $p$, then $\op{red}_\p(g(\t))=\overline{g}(z_0)\in\ff{1,\p}\subseteq \ff{\p}$.

If $N>0$, we look for the first index, $1\le i\le r+1$, for which the truncation $\op{Trunc}_i(\ty)$ does not divide $g(x)$. In the paragraph following Proposition \ref{vgt} we showed that this will always occur, eventually (for $i=r+1$) after improving the Montes approxi\-mation $\phi_\p=\phi_{r+1}$. By Proposition \ref{vgt}, there is a computable point $(s,u)\in N_i(g)$ such that
$v(g(\t_\p))=(sh_i+ue_i)/(e_1\cdots e_i)=v(\Phi_i(\t_\p)^s\pi_i(\t_\p)^u)$, and
$$
\op{lred}_\p\left(\dfrac{g(\t_\p)}{\Phi_i(\t_\p)^s\pi_i(\t_\p)^u}\right)=R_i(g)(z_i)\ne0.
$$
Now, if $(sh_i+ue_i)/(e_1\cdots e_i)>N$, we have $\op{lred}_\p(g(\t_\p)/p^N)=0$; on the other hand, if $(sh_i+ue_i)/(e_1\cdots e_i)=N$, we have
\begin{equation}\label{locred}
\op{lred}_\p\left(\dfrac{g(\t_\p)}{p^N}\right)= R_i(g)(z_i)\cdot \op{lred}_\p\left(\dfrac{\Phi_i(\t_\p)^s\pi_i(\t_\p)^u}{p^N}\right)
= R_i(g)(z_i)\, z_1^{t_1}\cdots z_i^{t_i},
\end{equation}
where $p^{-N}\Phi_i(\t_\p)^s\pi_i(\t_\p)^u=\gamma_1^{t_1}\cdots \gamma_i^{t_i}$, and the vector $(t_1,\dots,t_i)\in\Z^i$ can be found by the procedure of Lemma \ref{prodgammas}, applied to the input vector $$\log \left(p^{-N}\Phi_i(x)^s\pi_i(x)^u\right)=
(-N,0,\dots,0)+s\log \Phi_i+u\log \pi_i.$$
Since $\log \Phi_i$, $\log \pi_i$ have been stored as secondary invariants of $\ty_\p$,
we get in this way a really fast computation of $\op{lred}_\p(p^{-N}g(\t_\p))$.

This ends the computation of the reduction modulo $\p$ map $\op{red}_\p$.

\subsection{Chinese remainder theorem}
It is straightforward to design a chinese remainders routine once the following problem is solved.\medskip

\noindent{\bf Problem. }{\it Let $p$ be a prime number, $\P$ the set of prime ideals of $K$ lying above $p$, and $(a_\p)_{\p\in\P}$ a family of non-negative integers. Find a family $(c_\p)_{\p\in\P}$ of integral elements $c_\p\in\Z_K$ such that,
$$
c_\p\equiv 1 \md{\p^{a_\p}},\quad c_\p\equiv 0\md{\q^{a_\q}},\ \forall\,\q\in\P,\ \q\ne\p,
$$
for all $\p\in\P$.}\medskip

There is an easy solution to this problem: take the element $b_\p\in\Z_K$ satis\-fying (\ref{bp}), constructed in section \ref{subsectgenerators}, and consider $c_\p=(b_\p)^{p^t(q_\p-1)}$, where
$q_\p=\op{N}_{K/\Q}(\p)=\#\ff{\p}$, and $t$ is sufficiently large. However, the element $c_\p$ constructed in this way is not useful for practical purposes because it may have a huge norm and a huge height (very large numerators or denominators of the coefficients of its standard representation as a polynomial in $\t$), if $q_\p$ or $t$ are large. Instead, we shall refine the construction of the element $b_\p$ to get a small size solution $c_\p$ to the above problem.

First we deal with the particular case $a_\p=1$. The idea is to get an element $b_\p\in\Z_K$ satisfying
$$
v_\p(b_\p)=0,\quad v_\q(b_\p)\ge a_\q,\ \forall\,\q\in\P,\ \q\ne\p,
$$
and then find $\beta\in K$ such that $v_\p(\beta)=0$, $c_\p:=b_\p \beta$ is integral, and $c_\p\equiv 1\md{\p}$. Since we know the two-element representation $\q=(p,\alpha_\q)$ of the prime ideals, we could take $b_\p=\prod_{\q\in\P,\,\q\ne\p}(\alpha_\q)^{a_\q}$. Again, this might lead to a $b_\p$ with large size, so that a direct construction of $b_\p$ is preferable in order to keep its size as small as possible.

Thus, we consider an element $b_\p\in\Z_K$ as in (\ref{finalbp}):
$$
b_\p=\dfrac{(B_{[\p]}(x))^m\prod_{\l\in[\p],\,\l\ne\p}\phi_{\l}(x)^{d_{\l}}}{p^N},
$$
with $N=\sum_{\l\in[\p],\ \l\ne\p} d_\l v(\phi_\l(\t_\p))=\sum_{\l\in[\p],\ \l\ne\p} n_\l$. By construction, $v_\p(b_\p)=0$.

Let $i=\max_{\q\in\P, \,\q\ne\p}\{i(\ty_\p,\ty_\q)\}$; Lemma \ref{lessthanr} shows that $i\le r_\p$. From now on, we use only invariants of the type $\ty_\p$ and we drop the subindex $\p$ in the notation. Let
$$
M=\left\{\begin{array}{ll}
0,& \mbox{ if }i=0,\\
\left\lceil \dfrac{v_{i+1}(\phi_{i+1}))}{e_1\cdots e_i}\right\rceil,&\mbox{ if }i>0.
\end{array}
\right.
$$

Arguing as in section \ref{subsectgenerators}, we can take $m$ sufficiently large, and each $\phi_\l(x)$ sufficiently close to the $p$-adic irreducible factor $f_l(x)$, so that
\begin{equation}\label{plusM}
v_\q(b_\p)\ge a_\q+Me(\q/p), \ \forall \,\q\in\P,\ \q\ne\p,
\end{equation}
while keeping the denominator $p^N$ and the condition $v_\p(b_\p)=0$.
In particular, $b_\p$ belongs to $\Z_K$. The idea is to multiply $b_\p$ by some element in $K$ that conveniently modifies its residue class modulo $\p$.
We split this task into two parts, that may be considered as a kind of respective inversion modulo $\p$ of $(B_{[\p]}(\t))^m$ and $\,p^{-N}\prod_{\l\in[\p],\,\l\ne\p}\phi_{\l}(\t)^{d_{\l}}$.

Let $h(x)=(B_{[\p]}(x))^m$. Then, $\zeta:=\op{red}_\p(h(\t))\in\ff{1}\subseteq\ff{\p}$
is just the class of $h(x)$ modulo the ideal $(p,\phi_1(x))$. We invert $\zeta$ in $\ff{1}$ and represent the inverse $\zeta^{-1}=\overline{P}(z_0)$, as a polynomial in $z_0$ of degree less than $f_0$, with coefficients in the prime field $\ff{0}$. Take $\beta_0=P(\t)$, where  $P(x)\in\Z[x]$ is an arbitrary lift of $\overline{P}(x)$; clearly, $h(\t)\beta_0\equiv 1\md{\p}$.

Let now $\,g(x)=p^{-N}\prod_{\l\in[\p],\,\l\ne\p}\phi_{\l}(x)^{d_{\l}}$.
If $[\p]=\{\p\}$, then $i=0$, $N=M=0$, $g(x)=1$ and we are done. Suppose $[\p]\varsupsetneq \{\p\}$, so that $1\le i\le r_\p$. By the definition of the index of coincidence, we have  $\op{Trunc}_i(\ty_\p)\nmid \phi_\l(x)$, for all $\l\ne\p$; by the Theorem of the product \cite[Thm.2.26]{HN}, $\op{Trunc}_i(\ty_\p)\nmid g(x)$. Therefore, as we saw in the last section,
$$
\xi:=\op{lred}_\p(p^{-N}g(\t_\p))=R_i(g)(z_i)\,z_1^{t_1}\cdots z_i^{t_i}\in\ff{i+1},
$$
for some easily computable sequence of integers $(t_1,\dots,t_i)$. Let $V=e_1\cdots e_iM$; by \cite[Cor.3.2]{HN}, $v\left(\pi_{i+1}(\t_\p)^V\right)=M$. Compute a vector $(t'_1,\dots,t'_i)\in\Z^i$ such that
$p^{-M}\pi_{i+1}(x)^V=\gamma_1^{t'_1}\cdots \gamma_i^{t'_i}$, as indicated in Lemma \ref{prodgammas}, and take
$$
\xi':=z_1^{t'_1}\cdots z_i^{t'_i}\in\ff{i+1}.
$$
Let $\varphi(y)\in\ff{i}[y]$ be the unique polynomial of degree less than $f_i$, such that
$\varphi(z_i)=z_i^{\ell_iV/e_i}(\xi\xi')^{-1}$. Let $\nu=\ord_y\varphi(y)$.
Clearly,
$$V\ge v_{i+1}(\phi_{i+1})=e_if_iv_{i+1}(\phi_i),
$$the last equality by \cite[Thm.2.11]{HN}. Therefore, we can apply the constructive method described in \cite[Prop.2.10]{HN} to compute a polynomial $P(x)\in\Z[x]$ sa\-tisfying the following properties:
$$\dg P(x)<m_{i+1},\qquad v_{i+1}(P)=V, \qquad y^{\nu}R_i(P)(y)=\varphi(y).
$$
A look at the proof of \cite[Prop.2.10]{HN} shows that $N_i(P)$ is one sided of slope $\lambda_i$ and its end points have abscissa $e_i\nu$  and $e_i\deg\varphi$ (cf. Figure 4).

\begin{center}
\setlength{\unitlength}{5.mm}
\begin{picture}(10,5.4)
\put(5.85,1.85){$\bullet$}\put(2.85,3.35){$\bullet$}
\put(0,0){\line(1,0){10}}
\put(9,.5){\line(-2,1){8}}
\put(3,3.53){\line(2,-1){3}}
\put(.9,4.5){\line(1,0){.2}}
\put(1,-1){\line(0,1){6.5}}
\put(4.5,3){\begin{footnotesize}$N_i(P)$\end{footnotesize}}
\put(8.5,1){\begin{footnotesize}$L_{\lambda_i}$\end{footnotesize}}
\put(-.5,4.3){\begin{footnotesize}$V/e_i$\end{footnotesize}}
\multiput(3,-.1)(0,.25){15}{\vrule height2pt}
\multiput(6,-.1)(0,.25){9}{\vrule height2pt}
\put(5.2,-.6){\begin{footnotesize}$e_i\deg\varphi $\end{footnotesize}}
\put(2.6,-.6){\begin{footnotesize}$e_i\nu $\end{footnotesize}}
\end{picture}
\end{center}\medskip
\begin{center}
Figure 4
\end{center}

\noindent{\bf Claim: }$\op{lred}_\p(p^{-M}P(\t_\p))=\xi^{-1}$.\medskip

In fact, since $\dg P(x)<m_{i+1}$ and $v_{i+1}(P)=V$, the Newton polygon $N_{i+1}(P)$ is the single point $(0,V)$; in particular, $\op{Trunc}_{i+1}(\ty_\p)$ does not divide $P(x)$ and Proposition \ref{vgt}
shows that $v(P(\t_\p))=M$ and
$$
\op{lred}_\p\left(P(\t_\p)/\pi_{i+1}(\t_\p)^V\right)=R_{i+1}(P)(z_{i+1})\ne 0.
$$
Actually, $R_{i+1}(P)(y)$ has degree $0$ and it represents a constant in $\ff{i+1}$ that we denote simply by $R_{i+1}(P)$. By (\ref{locred}),
$$
\op{lred}_\p\left(p^{-M}P(\t_\p)\right)=R_{i+1}(P)\,\op{lred}_\p(p^{-M}\pi_{i+1}(\t_\p)^V)=R_{i+1}(P)\,\xi'.
$$
By the very definition of the residual polynomial \cite[Defs.2.20+2.21]{HN}, we have
$$
R_{i+1}(P)=z_i^{(e_i\nu -\ell_iV)/e_i}R_i(P)(z_i)=z_i^{-\ell_iV/e_i}\varphi(z_i)=(\xi\xi')^{-1},
$$
and the Claim is proven.\medskip

Finally, consider $$c_\p=b_\p\beta_0(p^{-M}P(\t)).$$ The condition (\ref{plusM}) ensures that $c_\p$ belongs to $\Z_K$ (although $p^{-M}P(\t)$ might not be integral) and satisfies $v_\q(c_\p)\ge a_\q$, for all $\q\in\P$, $\q\ne\p$. By construction, $\op{red}_\p(c_\p)=\op{lred}_\p(\iota_\p(c_\p))=1$.

It remains to solve our Problem when $a_\p>1$. In this case, we find $c\in\Z_K$ such that
$$
c\equiv 1\md{\p},\qquad c\equiv 0 \md{\q^{a_\q}},\ \forall\,\q\in\P,\ \q\ne\p,
$$
and we take $c_\p=(c-1)^m+1$, where $m$ is the least odd integer that is greater than or equal to $a_\p/v_\p(c-1)$.

\section{$p$-integral bases}\label{secBasis}
Let $p$ be a prime number and let $\ff{p}$ be the prime field of characteristic $p$. A \emph{$p$-integral basis of $K$} is a family of $n$ $\Z$-linearly independent integral elements $\alpha_1,\dots,\alpha_n\in\Z_K$ such that
$$
p\nmid \left(\Z_K\colon \gen{\alpha_1,\dots,\alpha_n}_\Z\right),
$$
or equivalently, such that the family $\alpha_1\otimes1,\,\dots,\,\alpha_n\otimes1$ is $\ff{p}$-linearly independent in the $\ff{p}$-algebra $\Z_K\otimes_\Z\ff{p}$.

If the discriminant $\dsc(f)$ of $f(x)$ may be factorized, the computation of an integral basis of $K$ (a $\Z$-basis of $\Z_K$) is based on the computation of $p$-integral bases for the different primes $p$ that divide  $\dsc(f)$.

Anyhow, even when $\dsc(f)$ may not be factorized, the computation of a $p$-integral basis of $K$ for a given prime $p$ is an interesting task on its own. In this section we show how to carry out this task from the data captured by the Okutsu-Montes representations of the prime ideals $\p$ lying over $p$.
For any such prime ideal we keep the notations for $f_\p(x)$, $\phi_\p(x)$, $\ff{\p}$, $\t_\p$, $K_\p$, $\Z_{K_\p}$, as introduced in section \ref{secPadic}.

\subsection{Local exponent of a prime ideal}Let $\P$ be the set of prime ideals lying over $p$. For any $\p\in\P$ we fix the topological embedding $K\hookrightarrow K_\p$ determined by sending $\t$ to $\t_\p$.

 \begin{definition}
We define the local exponent of $\p\in\P$ to be the least positive integer $\exp(\p)$ such that
$$
p^{\exp(\p)}\Z_{K_\p}\subseteq \Z_p[\t_\p].
$$
Note that $\exp(\p)$ is an invariant of the irreducible polynomial $f_\p(x)\in\Z_p[x]$, but it is not an intrinsic invariant of $\p$.
\end{definition}

The computation of $\exp(\p)$ is easily derived from the results of \cite{GMNokutsu}.

Let $\p=[p;\phi_1,\dots,\phi_r,\phi_\p]$ be the Okutsu-Montes representation of $\p$, as indicated in section \ref{subsecApprox}. By \cite[Lem.4.5]{GMNokutsu}:
$$f_\p(x)\equiv \phi_\p(x)\md{\m^{\lceil\nu\rceil}}, \quad \mbox{ where }\nu=\nu_\p+(h_{r+1}/e(\p/p)),
$$
and $\nu_\p$ is the rational number
$$\nu_\p:=\dfrac{h_1}{e_1}+\dfrac{h_2}{e_1e_2}+\cdots+\dfrac{h_r}{e_1\cdots e_r}.
$$
Caution: this number is not always an invariant of $f_\p(x)$. If $e_rf_r=1$, the Okutsu depth of $f_\p(x)$ is $R=r-1$ (cf. (\ref{depth})), and $h_r$ depends on the choice of $\phi_r$.

Denote $\phi_0(x)=x$, $m_0=1$, $m_{r+1}=n_\p:=e(\p/p)f(\p/p)$. Any integer $0\le m<n_\p$ can be written in a unique way as:
$$
m=\sum_{i=0}^ra_im_i, \quad 0\le a_i<\dfrac{m_{i+1}}{m_i}.
$$
Then, $g_m(x):=\prod_{i=0}^r\phi_i(x)^{a_i}$ is a divisor polynomial of degree $m$ of $f_\p(x)$ \cite[Thm.2.15]{GMNokutsu}. Also, if $\nu_m:=\lfloor v(g_m(\t))\rfloor $, the family
\begin{equation}\label{basis}
1,\,\dfrac{g_1(\t_\p)}{p^{\nu_1}},\,\dots,\,\dfrac{g_{n_\p-1}(\t_\p)}{p^{\nu_{n_\p-1}}}
\end{equation}
is a $\Z_p$-basis of $\Z_{K_\p}$ \cite[I,Thm.1]{Ok}. Since the numerators $g_m(x)$ have strictly increasing degree and $\nu_1\le \cdots\le\nu_{n_\p-1}$, it is clear that
$\exp(\p)=\nu_{n_\p-1}$.

On the other hand, since $m_{i+1}/m_i=e_if_i$, we have
$$\nu_{n_\p-1}=\left\lfloor \sum_{i=1}^r(e_if_i-1)v(\phi_i(\t_\p))
\right\rfloor
$$
Now, along the proof of \cite[Lem.4.5]{GMNokutsu}, it was proven that:
\begin{equation}\label{above}
\sum_{i=1}^r(e_if_i-1)v(\phi_i(\t_\p))=v(\phi_{r+1}(\t_\p))-\nu_\p-\dfrac{h_{r+1}}{e(\p/p)}=\dfrac{v_{r+1}(\phi_{r+1})}{e(\p/p)}-\nu_\p.
\end{equation}
Thus, if we combine
(\ref{above}) with the explicit formula:
$$
v_{r+1}(\phi_{r+1})=\sum_{i=1}^re_{i+1}\cdots e_r(e_if_i\cdots e_rf_r)h_i,
$$
given in \cite[Prop.2.15]{HN}, we get the following explicit computation of $\exp(\p)$ in terms of the invariants of the Okutsu-Montes representation of $\p$.

\begin{theorem}\label{theoexp}For all $\p\in\P$, we have $\exp(\p)=\lfloor\mu_\p\rfloor$, where
$$
\mu_\p:=\dfrac{v_{r+1}(\phi_{r+1})}{e(\p/p)}-\nu_\p=\sum_{i=1}^r(e_if_i\cdots e_rf_r-1)\dfrac{h_i}{e_1\cdots e_i}.
$$
\end{theorem}

Note that $\mu_\p$ is an Okutsu invariant, because it depends only on $e_i,f_i,h_i$, for $1\le i\le R$, where $R$ is the Okutsu depth of $f_\p(x)$. If $R=r-1$, then $e_rf_r=1$ and the summand of $\mu_\p$ corresponding to $i=r$ vanishes.

\subsection{Computation of a $p$-integral basis}
Along the computation of the prime ideal decomposition of the ideal $p\Z_K$ (by a single call to Montes algorithm) we can easily store the local exponents $\exp(\p)$, and the numerators $g_m(x)$ and denominators $p^{\nu_m}$ of all $\Z_p$-bases of $\Z_{K_\p}$ given in (\ref{basis}), for all $\p\in\P$.

It is well-known how to derive a $p$-integral basis from all these local bases. Let us briefly describe a concrete procedure to do this, taken from  \cite{ore2}.

We apply the method described at the end of section \ref{subsectgenerators} to compute multipliers $\{b_\p\}_{\p\in\P}$ satisfying:
\begin{equation}\label{multipliers}
v_\p(b_\p)=0,\quad v_\q(b_\p)\ge (\exp(\p)+1)\,e(\q/p),\ \forall\q\in\P,\,\q\ne\p.
\end{equation}
Consider the family obtained by multiplying each local basis (\ref{basis}) by its corresponding multiplier:
$$
\bb_\p:=\left[b_\p,\,b_\p\,\dfrac{g_1(\t)}{p^{\nu_1}},\,\dots,\,b_\p\,\dfrac{g_{n_\p-1}(\t)}{p^{\nu_{n_\p-1}}}\right].
$$
Then, $\bb:=\bigcup_{\p\in\P}\bb_\p$ is a $p$-integral basis of $K$.

In fact, although the elements $g_m(\t)/p^{\nu_m}$ are not (globally) integral, (\ref{multipliers}) shows that the products $\alpha_{m,\p}:=b_\p\,(g_m(\t)/p^{\nu_m})$ belong all to $\Z_K$ and satisfy
$$
v_\q(\alpha_{m,\p})\ge e(\q/p),\quad \forall \q\in\P,\, \q\ne\p.
$$
It is easy to deduce from this fact that the family of all $\alpha_{m,\p}$ determines an $\ff{p}$-linearly independent family of the algebra $\Z_K\otimes_\Z\ff{p}$.

\section{Some examples}\label{secKOM}
We have implemented  the algorithms described above in a package for Magma. The arithmetic of number fields in any algebraic manipulator has to face two problems: the factorization of the discriminant and the memory requirements for large degrees. Our package allows the user to skip the first problem, while it uses very little memory,  expanding Magma's capabilities by far. The package
which can be downloaded from its web page ({\tt http:/ma4-upc.edu/$\sim$guardia/+Ideals.html}),
is described in detail in the accompanying paper \cite{GMNPackage}. We include here a few examples  which exhibit the power of the package in different situations. More exhaustive tests of Montes algorithm have been presented in
\cite{GMNalgorithm}, \cite{GMNbasis}.

The computations in these examples have been done with Magma v2.15-11  in a Linux server, with two Intel Quad Core processors, running at 3.0 Ghz, with 32Gb of RAM memory.

\subsection{Large degree}   Consider the number field $K=\Q(\theta)$ given by a root  $\theta\in\overline\Q$ of the polynomial $f(x)=x^{1000} + 2^{50} x^{50} + 2^{60}$. The factorization of the discriminant of $f$ is
$$
\mbox{Disc}(f)=
2^{53940}3^{50}5^{2000}127^{50} 313^{50} 743^{50} 4886229527^{50}p^{50},
$$
with $p=337572698551220494882323528404563236947916489629537$.
The large degree of $f$ makes impossible to work in this number field using the standard functions of Magma, even after factorizing the discriminant, since the computation of the integral basis is necessary for these functions. But our algorithms avoid this computation, so that we can work with ideals in $K$. For instance,  in the table below we show the local index of the primes dividing $\operatorname{Disc}(f)$ and  the time taken to decompose them in $K$.
$$
\begin{array}{|c|r|r|}
\hline
\rm{Ideal} &\rm{Index} &\rm{Time}\\
\hline\hline
2\Z_K & 26235 &0.36s \\
\hline
3\Z_K & 0& 0.61s\\
\hline
5\Z_K &20& 0.63s\\
\hline
127\Z_K & 0& 1.29s\\
\hline
313\Z_K &0& 3.69s\\
\hline
743\Z_K &0& 6.47s \\
\hline
4886229527\Z_K&0 & 6.96s \\
\hline
p\Z_K &0& 60s\\
\hline
\end{array}
$$
Thus, we need less than 90 seconds to see that the discriminant of $K$ is
$$
\mbox{Disc}(K)=
2^{1470}3^{50}5^{1960}127^{50} 313^{50} 743^{50} 4886229527^{50}p^{50}.
$$
The running times in the table show clearly that the cost of the factorizations increases mainly because of the size of the numbers involved, and that the index has not a serious impact on them. The largest type appearing in these computations has order 3, and it appears along the factorization of the ideal $2\Z_K$, which is
$$
2\Z_K=\p_1^{10}({\p_1'})^{38}\p_4^{10}({\p_4'})^{38}\p_{20}^{38},
$$
where $\p_f^e$ stands for a prime ideal with residual degree $f$ and ramification index $e$. While we cannot expect to factor the ideals $I=(\t^3+50)\Z_K,$ $J=(\t+10)\Z_K$ in a reasonable time, it takes 0.03 seconds to compute the factorization of its sum:
$$
I+J=\p_1^{2}({\p_1'})^{2}\p_4^{2}({\p_4'})^{2}\p_{20}^{2}.
$$
The decomposition of 5 in the maximal order $\Z_K$ is
$$
5\Z_K=\p_2^5({\p'}_2)^5\p_{2}^{20}({\p'}_2)^{20}\p_2^{25}\p_4^{25}\p_{15}^{25}({\p_{15}'})^{25}.
$$
With the residue map computation explained in subsection \ref{redmap}, we may check very quickly that
$$
\theta\equiv \zeta\left(\operatorname{mod}{\p_4}\right),
$$
where $\Z_K/\p_4\simeq \mathbb{F}_5[\zeta]$, with $\zeta^4+2\zeta^2+3=0$. The Chinese remainder algorithm works also very fast in this number field.

\subsection{Small degree, large coefficients}

The space of modular forms of level 1 and degree 76 has dimension 6. The newforms in this space are defined over the number field $K=\Q(\theta)$, where $\theta\in\overline{\Q}$ is a root of the polynomial:
$$
\begin{array}{l}
f(x)=x^6 + 57080822040x^5 - 198007918566571424544768x^4 \\
\qquad- 11405115067164354385292006554337280x^3 \\
\qquad   + 9757628454131691442128845013041495838774263808x^2  \\
\qquad    +290013995562379500498435975003716024800114593761580810240x\\
   - 92217203874207784163935379997152082331434364841943058919508374716416.
\end{array}
$$
The discriminant of $f(x)$ is
$$
\begin{array}{rl}
\operatorname{Disc}(f)=&
2^{264} 3^{72} 5^{16} 7^{16} 11^2 13^2 17^4 19^2 43^2 59\cdot 193^2\cdot \\
&\qquad\qquad 293\cdot391987^2 4759427^2 137679681521^2M,
\end{array}
$$
where $M$ is a composite integer of 135 decimal figures which we have not been able to factorize. J. Rasmussen asked us (\cite{Rasmussen}) for a test to check certain divisibility conditions on the ring of integers of $K$, related to his work on congruences satisfied by the coefficients of certain modular forms. The time to find the decomposition of the primes in the set 
$$
S:=\{
2,3,5,7,11,13,17,19,43,59,193,293,391987,4759427,137679681521
\} 
$$
is almost negligible, since it involves only types of order at most 1. The table below shows the local indices of these primes:
$$
\begin{array}{|c|r|}
\hline
\rm{Ideal} &\rm{Index} \\
\hline
\hline
2 \Z_K &   132 \\  \hline
3 \Z_K &   36 \\  \hline
5 \Z_K &   8 \\  \hline
7 \Z_K &   8 \\  \hline
11 \Z_K &   1 \\  \hline
13 \Z_K &   1 \\  \hline
17 \Z_K &   2 \\  \hline
19 \Z_K &   1 \\  \hline
43 \Z_K &   1 \\  \hline
59 \Z_K &   0 \\  \hline
193 \Z_K &   1 \\  \hline
293 \Z_K &   0 \\  \hline
391987 \Z_K &   1 \\  \hline
4759427 \Z_K &   1 \\  \hline
137679681521 \Z_K &   1 \\  \hline
\end{array}
$$
Hence the discriminant of $K$ is $\operatorname{Disc}(K)=59\cdot 293N$, where $N$ is divisible by at least one of the prime factors of $M$, since $M$ is not a square.
The ideal prime decomposition of 3 in $\Z_K$ is
$$
3\Z_K=\p_2\p_1\p_1'\p_1''\p_1'''.
$$
The algorithm explained in section \ref{secGenerators} provides generators for all these ideals:
$$
\begin{array}{l}
\p_2=3\Z_K\!+3^{-12}(4\theta^5\! + 4311\theta^4\! + 1717038\theta^3\! + 2900691\theta^2 \!+ 820125\theta\!+ 2834352)\Z_K \\
\p_1=3\Z_K\!+3^{-11}(2\theta^5 + 1815\theta^4 + 586980\theta^3 + 732159\theta^2 + 658287\theta + 1535274)\Z_K\\
\p_1'=3\Z_K\!+3^{-11}(2\theta^5\! + 2031\theta^4 \!+ 662796\theta^3\! + 1123632\theta^2\! + 1071630\theta + 295245)\Z_K\\
\p_1''=3\Z_K\!+3^{-11}(2\theta^5 + 2307\theta^4 + 910872\theta^3 + 847584\theta^2 + 398034\theta + 1121931\Z_K\\
\p_1'''=3\Z_K\!+3^{-11}(2\theta^5 \!+ 2091\theta^4\! + 708696\theta^3 \!+ 646380\theta^2\! + 634230\theta + 1121931)\Z_K\\
\end{array}
$$
Applying the algorithm described in section \ref{secCRT}, we can compute without much effort  an element $\alpha\in K$ satisfying
$$
\begin{array}{lll}
\alpha\equiv 1(\operatorname{mod}\,{\p_2}),
\quad & \alpha\equiv \theta\left(\operatorname{mod}\,{\p_1}\right),\\\\
\alpha\equiv \theta^2\left(\operatorname{mod}\,({\p_1'})^2\right),
\quad &\alpha\equiv \theta^3\left(\operatorname{mod}\,({\p_1''})^3\right),
\quad &\alpha\equiv \theta^4\left(\operatorname{mod}\,({\p_1'''})^4\right).
\end{array}
$$
We may take, for instance:
$$
\alpha=3^{-9}(786086\theta^5 + 445989\theta^4 + 196857\theta^3 + 1159353\theta^2 + 649539\theta + 354294).
$$
Following the algorithm for $p$-adic valuations introduced in section \ref{secPadic}, we can check  this result  computing the valuations of the differences $r-\theta^j$ at the prime ideals dividing 3:
$$
v_{\p_2}(\alpha-1)=1,\
v_{\p_1}(\alpha-\theta)=7,\
v_{\p_1'}(\alpha-\theta^2)=4,\
v_{\p_1''}(\alpha-\theta^3)=4,\
v_{\p_1'''}(\alpha-\theta^4)=4.
$$

All these computations are almost immediate. Even the computation of an $S$-integral basis takes only 0.06 seconds.

\subsection{Medium degree}
Consider the polynomials:
$$\as{1.2}
\begin{array}{ll}
\phi_0= x+1,&\phi_1=\phi_0^{2}+2, \\
\phi_{21}=\phi_1^{2}+8, &\phi_{22}= \phi_1^{4}+4\phi_0\phi_1^{2}+32, \\
\phi_3=\phi_{22}^{2}+256\phi_1^{2},&
f=\phi_3\phi_{21}+2^{30}.
\end{array}
$$
Let $K=\Q(\theta)$ be the number field of degree 20 determined by a root $\theta\in\overline{\Q}$ of $f$. For the prime $p=2$,
the polynomial $f$ has two complete types, with associated Okutsu frames  $[\phi_1,\phi_{21}]$ and $[\phi_1,\phi_{22},\phi_3]$, which give rise to the two prime  ideals  of $\Z_K$ over $2$.
The concrete decomposition is $2\Z_K=\p_1^4\p_2^8$, where $f(\p_i/2)=i$, $e(\p_i/2)=4i$.

The discriminant of $f$ is
$$
\begin{array}{rl}
\operatorname{Disc}(f)=&2^{268}\cdot 3^2\cdot 19927\cdot 43691^2\cdot 211039\cdot 6059454913\cdot\\
&512920919154157817\cdot
 25506978885046388417449\cdot\\
 & 149169795543042282387542317948232968678925571739.
\end{array}
$$

In this example we may compare the performance of the standard Magma functions and that of our package, since Magma can determine the ring of integers of $K$. Once the factorization of $\operatorname{Disc}(f)$ is known, Magma takes 5.8 seconds to determine $\Z_K$, and 0.08 seconds to find the decomposition of the prime 2 in $\Z_K$.
Our package takes 0.3 seconds to see that $\operatorname{Disc}(K)=2^{-234}\operatorname{Disc}(f)$, and during this computation already finds the decomposition of all the primes dividing the discriminant.  Our program can also compute a  2-integral basis of $K$, which is already a global integral basis,  in 0.02 seconds.

\section{Conclusions}\label{secConclusion}
\subsection{Challenges}\label{first}
We described routines to perform the basic tasks concerning fractional ideals of a number field, based on the Okutsu-Montes representations of the prime ideals \cite{montes}, \cite{HN}. This avoids the factorization of the discriminant of a defining equation and the construction of the maximal order. These routines are very fast in practice, as long as one deals with fractional ideals whose norm may be factorized.

A big challenge arises: is it possible to combine these techniques with some kind of LLL reduction to test if a fractional ideal is principal?

Also, the generators of the prime ideals constructed in this paper have small height as vectors in $\Q^n$ (the coefficients of its standard representation as a polynomial in $\t$). This may have some advantages, but in many applications it is preferable to have generators of small norm. A solution to the above mentioned challenge would
probably lead to a procedure to find generators of small norm too.

%\subsection{Class group}

\subsection{Comparison with the standard methods}Suppose the discriminant of the defining equation of the number field may be factorized. Most of the methods to compute a $\Z$-basis of the maximal order are based on variants of the Round 2 and Round 4 algorithms of Zassenhaus.
The procedure of section \ref{secBasis} yields a much faster computation of an integral basis and the discriminant of the field.

Once the maximal order is constructed, we can compare our routines for the manipulation of fractional ideals with the standard ones. The routines based on the Okutsu-Montes representations of the prime ideals are faster, mainly because they avoid the usual linear algebra techniques (computation of bases of the ideals, Hermite and Smith normal forms, etc.), which become slow if the degree of the number field grows.

\subsection{Curves over finite fields}The results of these paper are easily extendable to function fields. If $C$ is a curve over a finite field, there is a natural identification of rational prime divisors of $C$ with prime ideals of the integral closures of certain subrings of the function field \cite{hess}, \cite{hess2}. Montes algorithm may be applied as well to construct these prime ideals, and the routines of this paper lead to parallel routines to find the divisor of a function, or to construct a function with zeros and poles of a prescribed order, at a finite number of places.

The results of section \ref{secBasis} may be used to efficiently compute bases of the above mentioned integral closures too. However, the big challenge of section \ref{first} has its parallel in the geometric situation: we hope that the techniques of this paper may be used to find better routines to compute bases of the Riemann-Roch spaces and to deal with reduced divisors. This would open the door to operate in the group $\op{Pic}^0(C)$ of rational points of the Jacobian of $C$, for curves with plane models of very large degree.

\end{document}